\documentclass[12pt]{amsart}
\usepackage{amsthm,amsmath,amssymb,amscd,graphics,enumerate, stmaryrd,xspace,verbatim, epic, eepic,color,url}

\usepackage[active]{srcltx}
\usepackage[all]{xypic}
\SelectTips{cm}{}

\newcommand{\marg}[1]{\normalsize{{\color{red}\footnote{{\color{blue}#1}}}{\marginpar[{\color{red}\hfill\tiny\thefootnote$\rightarrow$}]{{\color{red}$\leftarrow$\tiny\thefootnote}}}}}

\newcommand{\Barbara}[1]{\marg{(Barbara) #1}}

\newcommand{\ChDan}[1]{{\color{magenta} #1}}

{
   \newtheorem{theorem}[subsubsection]{Theorem}
      \newtheorem*{theorem*}{Theorem}
   \newtheorem{proposition}[subsubsection]{Proposition} 
   \newtheorem{prop}[subsubsection]{Proposition}     
   \newtheorem{lemma}[subsubsection]{Lemma}

   \newtheorem{corollary}[subsubsection]{Corollary}
   
   \newtheorem*{conjecture*}{Conjecture}
   
}
{\theoremstyle{definition}
          \newtheorem*{exercise*}{Exercise}

   \newtheorem*{example*}{Example}
   \newtheorem{definition}[subsubsection]{Definition}
   \newtheorem{defn}[subsubsection]{Definition}
   \newtheorem*{definition*}{Definition}
   
   \newtheorem{remark}[subsubsection]{Remark}
   \newtheorem{conv}[subsubsection]{Convention}
   \newtheorem{notation}[subsubsection]{Notation}
}
\newcommand{\RR}{{\mathbb{R}}}

\newcommand{\FF}{{\mathbb{F}}}
\newcommand{\CC}{{\mathbb{C}}}
\newcommand{\QQ}{{\mathbb{Q}}}
\newcommand{\NN}{{\mathbb{N}}}

\newcommand{\PP}{{\mathbb{P}}}
\newcommand{\ZZ}{{\mathbb{Z}}}
\newcommand{\EE}{{\mathbb{E}}}
\newcommand{\GG}{{\mathbb{G}}}
\newcommand{\LL}{{\mathbb{L}}}
\newcommand{\bbA}{{\mathbb{A}}}
\renewcommand{\AA}{{\mathbb{A}}}

\newcommand{\bL}{{\mathbf{L}}}

\newcommand{\bR}{{\mathbf{R}}}

\newcommand{\bfe}{{\mathbf{e}}}
\newcommand{\bff}{{\mathbf{f}}}
\newcommand{\bfc}{{\mathbf{c}}}
\newcommand{\bfd}{{\mathbf{d}}}

\newcommand{\bDelta}{{\boldsymbol{\Delta}}}

\newcommand{\bmu}{{\boldsymbol{\mu}}}

\newcommand{\bc}{{\mathbf{c}}}

\newcommand{\bone}{{\boldsymbol{1}}}

\newcommand{\fT}{{\mathfrak{T}}}
\newcommand{\fW}{{\mathfrak{W}}}
\newcommand{\fX}{{\mathfrak{X}}}

\newcommand{\fr}{{\mathfrak{r}}}

\newcommand{\cA}{{\mathcal A}}

\newcommand{\cC}{{\mathcal C}}
\renewcommand{\cD}{{\mathcal D}}
\newcommand{\cE}{{\mathcal E}}

\newcommand{\cG}{{\mathcal G}}

\newcommand{\cI}{{\mathcal I}}
\newcommand{\cK}{{\mathcal K}}

\newcommand{\cO}{{\mathcal O}}

\newcommand{\cT}{{\mathcal T}}

\newcommand{\cW}{{\mathcal W}}
\newcommand{\cX}{{\mathcal X}}
\newcommand{\cY}{{\mathcal Y}}
\newcommand{\cZ}{{\mathcal Z}}

\newcommand{\fC}{{\mathfrak C}}

\newcommand{\fM}{{\mathfrak M}}

\newcommand{\fQ}{{\mathfrak Q}}
\newcommand{\tw}{{\operatorname{tw}}}

\newcommand{\vir}{{\text{vir}}}
\def\<{\langle}
\def\>{\rangle}

\newcommand{\Spec}{\operatorname{Spec}}

\newcommand{\Sing}{{\operatorname{Sing}}}

\newcommand{\Hom}{{\operatorname{Hom}}}

\newcommand{\Ext}{{\operatorname{Ext}}}
\newcommand{\bExt}{{\operatorname{Ext}}}
\newcommand{\cExt}{{{\cE}xt}}

\newcommand{\Mor}{{\operatorname{Mor}}}

\newcommand{\Aut}{{\operatorname{Aut}}}

\newcommand{\Def}{{\operatorname{Def}}}
\newcommand{\sh}{{\operatorname{sh}}}

\newcommand{\lrar}{\longrightarrow}

\newcommand{\ocI}{\overline{{\mathcal I}}}

\newcommand{\lcm}{{\operatorname{lcm}}}

\newcommand{\double}{\genfrac..{0pt}1
{\raise -1pt\hbox{$\scriptstyle\longrightarrow$}}{\raise 3pt\hbox
{$\scriptstyle\longrightarrow$}}} 

\renewcommand{\setminus}{\smallsetminus}

\def\spl{{\rm spl}}

\def\tototi{\mathbin{\mathop{\otimes}\limits^{\raise-1pt\hbox
{$\scriptscriptstyle {\rm L}$}}}}

\def\indlim{\mathop{\vrule width0pt height7pt depth
4pt\smash{\lim\limits_{\raise 1pt\hbox to 14.5pt
{\rightarrowfill}}}}}
\def\projlim{\mathop{\vrule width0pt height7pt depth
4pt\smash{\lim\limits_{\raise 1pt\hbox to 14.5pt
{\leftarrowfill}}}}}

\newcommand\displaceamount{3pt}

\newcommand{\doubledown}{\ar@<\displaceamount>[d]\ar@<-\displaceamount>[d]}

\newcommand{\doubleup}{\ar@<\displaceamount>[u]\ar@<-\displaceamount>[u]}

\newcommand{\doubleright}{\ar@<\displaceamount>[r]\ar@<-\displaceamount>[r]}

\newcommand{\thickslash}{\mathbin{\!\!\pmb{\fatslash}}}

\newcommand{\eps}{\varepsilon}

\def\barM{\overline{M}}
\def\ev{\text{ev}}
\def\bev{{\mathbf{ev}}}

\newcommand{\fK}{\mathfrak K}
\newcommand{\tfK}{{\widetilde{\mathfrak K}}}
\def\id{\text{id}}
\def\bcI{\overline{\mathcal I}}
\newcommand{\sm}{{\operatorname{sm}}}

\def\Csi{\Xi}

\def\GGm{{\mathbb G}_m}

\def\st{\text{st}}

\begin{document}

\title[Orbifold techniques in degeneration formulas]{Orbifold techniques\\ in degeneration formulas}

\author[D. Abramovich]{Dan Abramovich}
\thanks{Research of D.A. partially supported by NSF grant DMS-0335501 and DMS-0603284}  
\address{Department of Mathematics, Box 1917, Brown University,
Providence, RI, 02912, U.S.A} 
\email{abrmovic@math.brown.edu}
\author[B. Fantechi ({\tiny\today})]{Barbara Fantechi}
\thanks{Research activities of B.F. partially supported by Italian research grant PRIN 2006 ``Geometria delle variet\'a proiettive", ESF Research Network MISGAM, Marie Curie RTN ENIGMA, and GNSAGA}  
\address{SISSA,
Via Bonomea 265,
34136 Trieste, Italy} 
\email{fantechi@sissa.it}

\date{\today, \jobname}


 \begin{abstract}  We give a new approach for relative and degenerate Gromov--Witten invariants, inspired by that of Jun Li but replacing predeformable maps by transversal maps to a twisted target. The main advantage is a significant simplification in the definition of the obstruction theory. We reprove in our language the degeneration formula, extending it to the orbifold case. 
\end{abstract}
\maketitle
\setcounter{tocdepth}{1}

\tableofcontents

 \section*{Introduction} 

\subsection{Gromov--Witten invariants in the smooth case}
Gromov--Witten invariants were originally defined for a compact symplectic manifold, and in the algebraic language for a smooth projective complex variety.  For an extensive bibliography see \cite{Cox-Katz}. From an algebraic viewpoint the construction proceeds via the following steps.

\begin{itemize}
\item[Step 1] Definition of a proper moduli stack $$\barM:=\barM_{g,n}(X,\beta)$$ of stable maps to $X$ with fixed discrete invariants $\beta\in H_2(X,\ZZ)$ and $g,n\in \NN$, together with an evaluation map $\ev:\barM\to X^n$. 
\item[Step 2] Construction on $\barM$ of a {\em $1$-perfect obstruction theory}, giving rise to a {\em virtual fundamental class} $[\barM]^\vir\in A_d(\barM)$, where $d$ is the {\em expected dimension} of $\barM$. 
\item[Step 3] Definition of the invariants  by integrating cohomology classes on $X^n$ against  $\ev_*((\prod \psi_i^{m_i})[\barM]^\vir)$.
\end{itemize}

This construction has subsequently been extended to the case of orbifolds, namely smooth Deligne--Mumford stacks, see \cite{Chen-Ruan, AGV}. For Step 1 above this case required care in the definition of stable maps, which were replaced by twisted stable maps in order to preserve properness of the moduli, see \cite{Abramovich-Vistoli}; on the other hand, the obstruction theory in Step 2 stayed essentially the same. The main difference in the formalism of Step 3 is that the evaluation map takes values in $\ocI(X)^n$, where $\ocI(X)$ is the so called  rigidified cyclotomic inertia stack of $X$. 

\subsection{Invariants of pairs and degenerate varieties}
If $X$ is singular, the moduli stack of stable maps is still proper; however, the natural obstruction theory is not perfect even for very mild singularities, and the construction has to be modified to stand a chance to work.

The issue was addressed  for singular varieties $W_0 = X_1 \sqcup_D X_2$, with $X_1,X_2$ and $D$ smooth appearing as fibers in a one-parameter family with smooth total space, by A.M. Li and Y. Ruan \cite{Li-Ruan}. It was also studied at about the same time by E. Ionel and T. Parker \cite{Ionel-Parker1,Ionel-Parker2}, and subsequently worked out in the algebraic language by Jun Li \cite{Li1,Li2}.
Here the moduli of stable maps was changed in such a way as to have a perfect obstruction theory while keeping properness.  With similar techniques, relative Gromov--Witten invariants were defined for a pair $(X,D)$ with $X$ a smooth projective variety and $D$ a smooth divisor in $X$. The main tool introduced  here  is that of {\em expanded degenerations and expanded pairs}. 

The {\em degeneration formula} is a way to express the Gromov--Witten invariants of $W_0=X_1\sqcup_DX_2$ in terms of the relative invariants of the pairs $(X_i,D)$. It is now a key tool in Gromov--Witten theory.

There is also work preceding the cited papers where the ideas involved appear in different guises. The idea of expanded degenerations and its use in enumerative geometry was introduced by Z. Ran \cite{Ran}. Even earlier Harris and Mumford introduced the related idea of admissible covers  \cite{Harris-Mumford}, revisited using logarithmic geometry by  Mochizuki \cite{Mochizuki}. Related ideas with a different view can be found in \cite{Hirschowitz}, \cite{Alexander-Hirschowitz}, \cite{Caporaso-Harris}, \cite{Bernstein}, \cite{Vakil}. A simple approach in special but important cases was developed by Gathmann \cite{Gathmann}.

\subsection{The twisting method}
In this paper we will give an alternative algebraic definition of Gromov--Witten invariants for singular varieties as above, and of relative Gromov--Witten invariants, which extends naturally also to Deligne--Mumford stacks. 
Our treament follows closely that of Jun Li. However, introducing appropriate auxiliary  orbifold  structures along the nodes of both source curves and target varieties allows us to give a shorter definition of the  the obtruction theory in Step 2 above, and a streamlined proof of the degeneration formula. At the same time we obtain a somewhat more general result, which applies to the orbifold case, see Theorem \ref{maintheorem} below.
\ChDan{We note that a symplectic geometry approach to the orbifold case was developed by B. Chen, A.-M. Li, S. Sun and G. Zhao in \cite{CLSZ}.}

In a nutshell, the most difficult point in Jun Li's approach is to define an obstruction theory on predeformable maps. A predeformable map $C \to W_0$ from a nodal curve $C$ to a variety $W_0$ with codimension-1 nodal singularities locally looks like
\[\begin{array}{c}
 \xymatrix{ 
C=\Spec \frac{\CC[u,v]}{(uv)}\ar[r] & \Spec \frac{\CC[x,y,z_i]}{(xy)} = W_0} 
\\
 \xymatrix{ 
u^c&& x\ar@{|->}[ll]} 
\\
 \xymatrix{ 
v^c&& y,\ar@{|->}[ll]
}\end{array}
\]
and $z_i \mapsto f_i(u,v)$ arbitrary.
As soon as $c$, the {\em contact order}, is $>1$, this predeformability condition is not open on maps but rather locally closed. This means that deformations and obstructions as predeformable maps cannot coincide with deformations and obstructions as maps, so an obstruction theory must be constructed by other means. Jun Li does this by a delicate explicit construction.

Our approach to this is the following: we replace $W_0$ by the orbifold $\cW_0 = [\Spec\frac{\CC[\xi,\eta,z_i]}{(\xi\eta)} / \bmu_c]$ having $W_0$ as its coarse moduli space. Here $\bmu_c$ acts via $(\xi,\eta) \mapsto (\zeta_c\xi,\zeta_c^{-1}\eta)$, and $x=\xi^c, y=\eta^c$.  Then the map $C \to W_0$ locally lifts to  
\[\begin{array}{c}
 \xymatrix{ 
C\ar[r] & \cW_0} 
\\
 \xymatrix{ u& \xi\ar@{|->}[l]}
\\
 \xymatrix{ v& \eta.\ar@{|->}[l]
}
\end{array}
\]
Transversal maps of this type form an open substack of all maps, so an obstruction theory is immediately given by the natural obstruction theory of maps. An identical twisting construction applies in the case of pairs.

This in itself works well when we look at one node of $C$ mapping to a singular locus of $W_0$. When several nodes $p_j$ map to the same singular locus, they may have different contact orders $c_j$. If we pick an integer $r$ divisible by all $c_j$, take $\cW_0 = [\Spec\frac{\CC[\xi,\eta,z_i]}{(\xi\eta)} / \bmu_r]$, and at each $p_j$ put a similar orbifold structure on $C$ with index $r/c_j$, we still obtain a transversal map and therefore a good obstruction theory. In order to keep the moduli stacks separated, we must sellect a way to choose the integer $r$. We do this using the notion of a {\em twisting choice} - a rule that assigns to a collection $\bc=\{c_j\}$ of contact orders a positive integer $r=\fr(\bc)$ divisible by all the contact orders $c_j$, see Definition \ref{D:twistingchoice}. 

With this at hand we can define Gromov--Witten invariants. Theorem \ref{gw-ind-twisting} shows that our invariants are independent of the twisting choice. Theorem \ref{T:definv} shows they are defomation invariants.

\subsection{The degeneration formula}

\begin{theorem}\label{maintheorem}
\begin{align*}
 \lefteqn{\left\langle{ \prod_{i=1}^n\tau_{m_i}(\gamma_i)}\right\rangle_{\beta,g}^{W_0} } \\
   &=&
 \sum_{\eta\in \Omega} \frac{\prod_{j\in M} d_j}{|M|!} \sum_{\substack{\delta_j\in F\\ \text{ for }j\in M}} (-1)^{\epsilon}&
 \left\langle \prod_{i\in N_1} \tau_{m_i}(\gamma_{i})\bigg\vert\prod_{j\in M}\delta_j
  \right\rangle^{(X_1,D)}_{\Csi_1}\\&&&\cdot
 \left\langle 
 \prod_{i\in N_2} \tau_{m_i}(\gamma_{i})\bigg\vert\prod_{j\in M}\tilde\delta^\vee_j
 \right\rangle^{(X_2,D)}_{\Csi_2}.
\end{align*}
\end{theorem}
\newcounter{myparagraph}[subsubsection]
\subsubsection{User's guide - left hand side}
\paragraph{\arabic{myparagraph}}\addtocounter{myparagraph}{1}
$W_0$ is a proper Deligne--Mumford stack having projective coarse moduli space $\bar W_0$. The rigidified inertia stack of $W_0$  is denoted $\ocI(W_0)$.
\paragraph{\arabic{myparagraph}}\addtocounter{myparagraph}{1}
 $W_0 = X_1\sqcup_DX_2$ has first-order smoothable singular locus $D$ separating it in two smooth stacks $X_1,X_2$, see Sections \ref{onpairs} and \ref{logsmpair}.
\paragraph{\arabic{myparagraph}}\addtocounter{myparagraph}{1}
$g\geq 0$ is an integer and $\beta$ is a curve class on $W_0$ (see \ref{S:Stablemaps-conventions}).
\paragraph{\arabic{myparagraph}}\addtocounter{myparagraph}{1}
$\gamma_1,\ldots,\gamma_n\in H^*_{orb}(W_0,\QQ) := H^*(\ocI(W_0),\QQ)$ are  classes having homogeneous parity, see opening of Section \ref{S:Degeneration}. In particular only classes on sectors transversal to $D$ are relevant.
\paragraph{\arabic{myparagraph}}\addtocounter{myparagraph}{1}
$m_1,\ldots,m_n\geq 0$ are integers. 
\paragraph{\arabic{myparagraph}}\addtocounter{myparagraph}{1}
Consider a twisting choice $\fr$ (see \ref{D:twistingchoice}) and the moduli stack $\cK :=\cK^\fr_\Gamma(W_0)$ of $\fr$-twisted stable maps (see \ref{D:rtsm}). 
\paragraph{\arabic{myparagraph}}\addtocounter{myparagraph}{1}
$\cK$ carries several universal maps, the coarsest of which is a stable map $\underline C \to \bar W_0$ from the coarse contracted curve $\underline C$ to the coarse target $\bar W_0$. We have $n$ sections $s_i: \cK \to \underline C$. We denote $\psi_i = s_i^*c_1(\omega_{\underline C/\cK})$.
\paragraph{\arabic{myparagraph}}\addtocounter{myparagraph}{1}
We have $n$ evaluation maps $\ev_i: \cK \to \ocI(W_0)$.
\paragraph{\arabic{myparagraph}}\addtocounter{myparagraph}{1}
Finally we define
$$\left\langle{ \prod_{i=1}^n\tau_{m_i}(\gamma_i)}\right\rangle_{\beta,g}^{W_0}\ \   = \ \   \deg\left(\left(\prod_{i=1}^n \psi_i^{m_i}\cdot \ev_i^*\gamma_i\right)\cap{[\cK]^\vir}\right).$$
\subsubsection{Right hand side}
\paragraph{\arabic{myparagraph}}\addtocounter{myparagraph}{1}
$F$ is a homogeneous basis for $H^*(\ocI(D),\QQ)$.
\paragraph{\arabic{myparagraph}}\addtocounter{myparagraph}{1}
$\tilde\delta^\vee$ is the dual of  $\delta\in F$ with respect  to the Chen--Ruan pairing, i.e. $\int_{\ocI(D)}\frac{1}{r}\iota^*\tilde\delta_j^\vee\cdot\delta_i  = \int_{\cI(D)}\iota^*\tilde\delta_j^\vee \cdot \delta_i = \delta_{i,j},$ see \ref{S:pf-mainth}. 
\paragraph{\arabic{myparagraph}}\addtocounter{myparagraph}{1}
$\Omega$ is the set of splittings of the data $g,n,\beta$, see \ref{D:enh-spl} for all details. An element $\eta=(\Csi_1,\Csi_2)\in \Omega$ includes in particular the data below:
\paragraph{\arabic{myparagraph}}\addtocounter{myparagraph}{1}
$N_1,N_2$ is a decomposition of $\{1,\ldots,n\}$ in two subsets.
\paragraph{\arabic{myparagraph}}\addtocounter{myparagraph}{1}
$\Csi_1,\Csi_2$ is a possibly disconnected splitting of the data $\beta,g$ in two modular graphs
having roots labelled by $M=\{n+1,\ldots,n+|M|\}$, see \ref{D:modular-graph}.
\paragraph{\arabic{myparagraph}}\addtocounter{myparagraph}{1}
$d_i$ for $i\in M$ are assigned intersection multiplicities satisfying condition B of \ref{D:enh-spl}
\paragraph{\arabic{myparagraph}}\addtocounter{myparagraph}{1}
We take $\cK_1 :=\cK^\fr_{\Csi_1}(X_1,D)$ and $\cK_2:=\cK^\fr_{\Csi_2}(X_2,D)$.
\paragraph{\arabic{myparagraph}}\addtocounter{myparagraph}{1}
For $j\in M$ the new evaluation maps are $\ev_j:\cK_1\to \ocI(D)$, and similarly for $\cK_2$.
\paragraph{\arabic{myparagraph}}\addtocounter{myparagraph}{1}
$(-1)^\epsilon$ is the sign determined formally by the equality
$$\prod_{i=1}^n \gamma_i\prod_{j\in M} \delta_j\tilde\delta_j^\vee = (-1)^\epsilon\prod_{i\in N_1} \gamma_i\prod_{j\in M} \delta_j \prod_{i\in N_2} \gamma_i\prod_{j\in M}   \tilde\delta_j^\vee.$$
\paragraph{\arabic{myparagraph}}\addtocounter{myparagraph}{1}
Finally we define
\begin{align*} 
\lefteqn{\left\langle \prod_{i\in N_1} \tau_{m_i}(\gamma_{i})\bigg\vert  \prod_{j\in M}\delta_j
  \right\rangle^{(X_1,D)}_{\Csi_1} } \\ & \ \  :=  \ \ 
\deg\left(\left(\prod_{i\in N_1} \psi_i^{m_i}\cdot \ev_i^*\gamma_i\right)\left(\prod_{j\in M}\ev_j^*\delta_j\right)\cap {[\cK_1]^\vir}\right) \\ 
\intertext{and similarly}
\lefteqn{\left\langle \prod_{j\in N_2} \tau_{m_j}(\gamma_{j})\bigg\vert\prod_{j\in M}\tilde\delta^\vee_j
  \right\rangle^{(X_2,D)}_{\Csi_1} } \\ & \ \  := \ \ 
\deg\left(\left(\prod_{i\in N_2} \psi_i^{m_i}\cdot \ev_i^*\gamma_i\right)\left(\prod_{j\in M}\ev_j^*\tilde\delta_j^\vee\right)\cap{[\cK_2]^\vir}\right).
\end{align*}

Perhaps the most mysterious part of the formula is the factor $\prod_{j\in M} d_j$. In previous works this arises as a result of delicate deformation theory of admissible or predeformable maps. In this paper it arises as a natural, but still delicate, outcome of the geometry of orbifold maps, see Lemma \ref{L:eval-lift-c} and Proposition \ref{P:diag-gysin}.

\subsection{The symplectic approach} \ChDan{A degeneration formula for symplectic orbifolds was worked out in  \cite{CLSZ}, which relies on orbifold good maps in the sense of Chen and Ruan and on analytic techniques. We have compared our formula with that of  \cite{CLSZ} and have been convinced that the results coincide, although the formalisms are sufficiently different that a direct comparison would be very technical.} 

\subsection{The logarithmic approach}
\ChDan{Jun Li's study of predeformable maps and their obstruction theory was inspired by logarithmic structures. Recently a more direct use of logarithmic structures has become possible, relying on
 the work of Olsson \cite{Olsson-log-cotangent}. B. Kim \cite{Kim} replaced J. Li's obstruction theory by one induced by the natural log structures, with a degeneration formula worked out  by Q. Chen in \cite{Chen-deg} and virtual localization worked out by Molcho and Routis \cite{Molcho-Routis}.  M. Gross and B. Siebert \cite{GS}, Chen \cite{Chen} and Abramovich and Chen \cite{AC} have developed a logarithmic theory without expansions. 

These approaches have been compared in \cite{AMW}, where it was shown that Li's invariants, the invariants introduced here, Kim's invariants, and the logarithmic invariants defined without expansions all coincide.}

\subsection{Outline of the paper}

In section \ref{S:Roots} we review twisted curves and root constructions, as their fine structure is key to our methods.

In Section \ref{S:Expansions} we briefly review expanded pairs and expanded degenerations and their twisted versions, and describe their boundary. In addition we introduce a weighted version of the stacks of twisted expanded pairs and expanded degenerations. Finally we treat stable configurations of points on  expanded pairs and expanded degenerations.

Section \ref{S:Stablemaps} leads to the construction of stacks of $\fr$-twisted stable maps and a proof of their properness. 

In Section \ref{S:Invariants} we define Gromov--Witten invariants using $\fr$-twisted stable maps. We prove their deformation invariance and independence of twisting choice.

The degeneration formula is restated and proven in Section \ref{S:Degeneration}. 

Results in sections \ref{S:Invariants} and \ref{S:Degeneration} rely on compatibility results for virtual fundamental classes. We find it useful to systematically use Costello's result  \cite[Theorem 5.0.1]{Costello} and its generalization in \cite[Proposition 2, Section 4.3]{Manolache} where a smoothness assumption is removed.

In addition we have three appendices, with necessary material which the knowledgeable reader may only wish to peruse when needed. In Appendix \ref{A:pairs}  we review material concerning pairs $(X,D)$, nodal singularities, and transversality. Appendix \ref{A:stacks} is devoted to a number of basic construction with stacks. In Appendix \ref{obs_th} we review the algebraicity of stacks of maps and construction and properties of their obstruction theories.


\subsection{Conventions}

The following conventions will be in force throughout the paper.

We work over an algebraically closed base field of characteristic $0$, denoted by $\CC$. We note that the assumption that the field be algebraically closed is mostly for convenience, whereas the characteristic assumption is significantly harder, and for some purposes impossible, to avoid.

\ChDan{Further, as soon as we study Gromov--Witten invariants, we require the field $\CC$ to be the field of complex numbers. This allows us to access singular cohomology, in which a K\"unneth decomposition of the diagonal of $D$ is available. One can avoid this by stating a cycle-theoretic degeneration formula as in \cite{Li2}, but we do not pursue this here.}

All stacks and morphisms are assumed to be locally of finite type over $\CC$, unless otherwise specified.

A point in a scheme or algebraic stack (sometimes denoted by $p\in X$) is a $\CC$-valued point, unless otherwise specified. 

Whenever we say locally we always mean \'etale locally, unless otherwise specified.

If $X$ is an algebraic stack, by $D(X)$ we denote the derived category of sheaves of $\cO_X$-modules with coherent cohomology. An object $\FF\in D(X)$ is called {\em perfect of perfect amplitude contained in} $[a,b]$, or just {\em perfect in} $[a,b]$ for brevity, if it is locally isomorphic to a complex of locally free sheaves in degrees $a,a+1,\ldots,b$.

An element in a skew commutative graded ring is of {\em homogeneous parity} if it is a sum of only even-degree or only odd-degree terms, in which case its parity is even or odd, respectively.

\subsection{Notation}\ 
\hfill\\

\noindent
\begin{tabular}{ll}
$\CC$ & fixed algebraically closed field of characteristic 0.\\
$\cA$ & The stack $[\bbA^1/\GG_m]$.\\
$\LL_{X/Y}$, $\LL_f$ & cotangent complex of a morphism $f:X \to Y$.\\
$k$ & number of components in an  \ChDan{expansion} \\
$\ell$ & generic splitting divisor\\
$r$ & twisting index along a divisor\\
${\bf r}$, $\fr$& twisting sequence on an  \ChDan{expansion}, twisting choice\\
$n_X$, $e_i$ & number of legs to general point, twisting tuple\\
$n_D$, $f_j$ & number of legs to boundary divisor, twisting tuple\\
$c_j$, $d_j$ & contact order and intersection multiplicity at such a point\\
$d$ & $\beta\cdot D$\\
$m$ & index of $\tau$ in descendant notation $\tau_m(\gamma)$\\
$h$, $j$ & number of components of a disconnected graph, their index 
\end{tabular}

\subsection*{Acknowledgements} This  work was initiated during a joint visit at the University of Nice.  The program was further developed in countless visits at the MFO, Oberwolfach, and a program at the de Giorgi Center, Pisa. It was finalized at a program in MSRI, which provided time and immediate interaction with numerous colleagues. We heartily thank Angelo Vistoli for numerous comments as the project evolved, and  Martin Olsson who jump-started our understanding of logarithmic structures.   Our indebtendness to Jun Li is evident throughout the paper. We also thank Tom Bridgeland, Charles Cadman, \ChDan{Bohui Chen}, Brian Conrad, Bumsig Kim,  Joseph Lipman, Amnon Ne'eman,  Ravi Vakil, Jonathan Wise and Aleksey Zinger for helpful discussions.

\section{Roots and twists}\label{S:Roots}

The aim of this section is to briefly outline the definition and properties of several stacks and universal families parametrizing twisted versions of the moduli spaces used in Jun Li's original proof. The starting point comes from the stack of twisted pre stable curves, which was introduced in \cite{Abramovich-Vistoli,AGV} to define GW invariants of orbifolds: the only difference with the usual pre stable curves is that it is allowed to add a twisted, or stacky, structure at the marked points and nodes. In each case the automorphism group is finite cyclic, an in the case of nodes it is also required to be balanced, a necessary and sufficient condition for the twisted curves to be smoothable. 

We apply the same principle to Jun Li's moduli stacks of expanded pairs and expanded degenerations, where we twist both the boundary divisors and the singular locus. We suggest that the reader skim through the definitions, going back to them and to the properties as needed in the course of reading the paper.

\subsection{Root stacks}\label{SS:root}
We review here the theory developed in \cite[Appendix B]{AGV}, \cite[Section 2]{Cadman1}, \cite[Section 4]{Matsuki-Olsson}.

\subsubsection{Root stack of a line bundle}
Let $X$ be an algebraic stack, $L$ a line bundle on $X$, and $r>0$ an integer. We define a stack $\sqrt[r]{L/X}$ by requiring that it parametrizes $r$-th roots of the pullback of $L$: objects of $\sqrt[r]{L/X}$ over a base scheme $S$ are triples $(f,M,\phi)$ where
\begin{enumerate}
\item $f:S \to X$ is a morphism, i.e. an object of $X(S)$,
\item $M$ is an invertible sheaf on $S$, and
\item $\phi:M^r \to f^*L$ is an isomorphism.
\end{enumerate}
An isomorphism between $(f,M,\phi)$ and $(\bar f, \bar M,\bar\phi)$ over the identity of $S$ is an isomorphism $\alpha:M\to \bar M$ of line bundles such that $\phi=\bar\phi\circ \alpha^{\otimes r}$.
Note that there is a natural isomorphism $\sqrt[r]{L/X}\to \sqrt[r]{L^\vee/X}$ defined by mapping $(f,M,\phi:M^{\otimes r}\to f^*L)$ to $(f, M^\vee, (\phi^\vee)^{-1})$.

\begin{prop}\label{rootlb} The stack $\sqrt[r]{L/X}$ is algebraic, and it is a gerbe banded by $\mu_r$ over $X$. In particular, the structure map $\sqrt[r]{L/X}\to X$ is \'etale and proper.
\end{prop}

Note that  trivialising the gerbe, i.e. giving an isomorphism $\sqrt[r]{L/X}\to X\times B\mu_r$, is equivalent to choosing a line bundle $M$ on $X$ together with an isomorphism $M^{\otimes r}\to L$.

\subsubsection{Root of a divisor}
Let $X$ be an algebraic stack and $D$ an effective  Cartier divisor on $X$; it defines a line bundle $\cO_X(D) = \cI_D^\vee$ with a canonical section $\bone_D$. We denote by  $X(\sqrt[r]{D})$ the stack parametrizing simultaneous $r$-th roots of the pullback of $\cI_D^\vee$ and of the section: objects over a base scheme $S$ are tuples $(f,M,\phi,s)$ where
\begin{enumerate}
\item $f:S \to X$ is a morphism,
\item $M$ is an invertible sheaf on $S$, 
\item $\phi:M^r \to f^*L$ an isomorphism, and
\item $s\in H^0(M)$ a section,
\end{enumerate}
such that $\phi(s^r) = \bone_D$.
Again arrows are given using pullbacks.

	Notice that in case $f:S\to X$ is a morphism \ChDan{ such that $f^*D$ is still a Cartier divisor,} then the groupoid $X(\sqrt[r]{D})(S)$ is rigid, and each isomorphism class defines a Cartier divisor $\cD:=Z(s)$ on $S$ such that $m\cD=f^*D$.

\begin{prop}\label{rootdiv} The stack $X(\sqrt[r]{D})(S)$ is algebraic, and the structure morphism $f:X(\sqrt[r]{D})(S)\to X$ is flat of relative dimension zero, and an isomorphism away from $D$. As a stack  the divisor $\cD$ is isomorphic to $\sqrt[r]{N_{D\subset X}/D}$.  \end{prop}

If we choose a line bundle $M$ on $X$ together with an isomorphism $M^\otimes r\to \cO_X(D)$, we can associate to it a simple cyclic cover $Y\to X$ branched over $D$; in this case, $X(\sqrt[r]{D})$ is isomorphic to the stack quotient $[Y/\mu_r]$. From this it follows that if
$(X,D)$ is a locally smooth pair in the sense of \ref{onpairs} (i.e., $D$ is smooth and $X$ is smooth near $D$) then so is $(X(\sqrt[r]D),\cD)$. In this case we sometimes say that $(X(\sqrt[r]D),\cD)$ is obtained from $(X,D)$ by twisting $D$ to order $r$.

\subsubsection{Roots with several divisors} Given finitely many  effective Cartier divisors  $D_1,\ldots,D_k$ on $X$, and given positive integers $r_1,\ldots ,r_k$, we use the following notation:

$$X(\sqrt[r_1]{D_1},\ldots,\sqrt[r_k]{D_k})\ \  :=\ \  X(\sqrt[r_1]{D_1})\times_X\cdots\times_X X(\sqrt[r_k]{D_k}).$$

We will mostly use this notation when $(X,D_i)$ is a locally smooth pair for every $i$ and the divisors $D_i$ meet transversally: in this case, the same is true for $X(\sqrt[r_1]{D_1},\ldots,\sqrt[r_k]{D_k})$ and the $\cD_i$.

\subsubsection{Comparison of roots}\label{SSS:compare-roots} Note that if $r = r'\cdot r''$ then $\nu_r = \nu_{r'} \circ \nu_{r''}$. In particular we have canonical morphisms
$\sqrt[r]{L/X} \to \sqrt[r']{L/X}$ and $X(\sqrt[r]D)\to X(\sqrt[r']D)$. In fact this gives a canonical isomporphism 

$$\xymatrix{\sqrt[r]{L/X}\ar[r]& \sqrt[r'']{M \big/ \left(\sqrt[r']{L/X}\right)}
}$$ 
and similarly an isomorphism
$$\xymatrix{X(\sqrt[r]D)\ar[r] & \left(X(\sqrt[r']D)\right)\left(\sqrt[r'']{\cD}\right).}$$ 

\subsubsection{Twisted curves as root stacks: the markings} Suppose now $\cC$ is a twisted curve with markings $\Sigma_i$ with indices $e_i$. Then $\cC$ is canonically a root stack as follows. Consider the curve $\cC^u$ obtained by gluing the coarse moduli space of the smooth locus $\cC^\sm$ with $\cC \setminus (\cup \Sigma_i)$. We have a ``partial coarse moduli space" morphism $\pi':\cC \to \cC^u$. Denote by $\Sigma_i^u$ the markings on $\cC^u$. Then ${\pi'}^* \Sigma_i^u = r_p \Sigma_i$. This gives a canonical morphism $$\cC \to \cC^u(\sqrt[e_1]{\Sigma_1},\ldots,\sqrt[e_n]\Sigma_n)$$ which is easily seen to be an isomorphism \cite[Theorem 4.2.1]{AGV}. 

Generalizing the structure of twisted curves at nodes is a bit more subtle, see \ref{balanced_node}.

\subsubsection{Triviality of relative automorphisms} 
We return to the general setup, and consider 
Since $X(\sqrt[r]D) \to X$ is representable over the dense open $X\setminus D$, the groupoid $\Aut_X(X(\sqrt[r]D))$ is equivalent to a group, see \ref{2stacks_as_stacks}, and we regard it as a group. But since for  dominant $f:S \to X$ an object $(f,M,\phi,s)$ is determined by $f$, the group $\Aut_X(X(\sqrt[r]D))$ is trivial.


\subsubsection{Inertia of root stacks when $X$ is a scheme} Inertia stacks are reviewed in \ref{SS:inertia}. Since we are working over $\CC$ we will identify inertia stacks and cyclotomic inertia stacks.
It will be useful for us to describe the  cyclotomic inertia stack  of $X(\sqrt[r]{D})$ and its rigidified version, and similarly for the substack $\cD$. The picture is clear when $X$ is a scheme or an algebraic space: since $\cD$ is a gerbe we have $\cI(\cD) = \sqcup_{i=0}^{r-1}\cD$ and  $\ocI(\cD) = \sqcup_{i=o}^{r-1}\cD_i$, where $\cD_i\simeq \sqrt[g]{N_{D\subset X}/D}$ and $g=\gcd(r,i)$. We similarly have
$$\cI(X(\sqrt[r]{D})) = X(\sqrt[r]{D}) \sqcup \coprod_{i=1}^{r-1}\cD,$$ and 
$$\ocI(X(\sqrt[r]{D})) = X(\sqrt[r]{D}) \sqcup \coprod_{i=1}^{r-1}\cD_i.$$
 The latter follows from the decompositions into an open substack  and closed complement 
\begin{align*}\cI(X(\sqrt[r]{D})) &\ \ \ =\ \ \  (X\setminus D)\ \sqcup\ \cI(\cD)\\ \intertext{and} \ocI(X(\sqrt[r]{D}))&\ \ \ =\ \ \  (X\setminus D)\ \sqcup\ \ocI(\cD).
\end{align*}


\subsubsection{Inertia of root stacks when $X$ is an orbifold} In case $X$ itself is an orbifold the picture is almost identical, using the inertia stacks of $X$ and $D$: we still have decompositions into an open substack and closed complement precisely as above. The coarse moduli space of the stack $\cI(\cD)$ consists of  $r$ copies of the coarse moduli space of $\cI(D)$. However the stack structure of the components  of $\cI(\cD)$ and $\ocI(\cD)$ becomes slightly more involved than in the case when $X$ is representable.

We note that $\cI(B\GG_m)\simeq\GG_m \times B\GG_m$, and the morphism $\cI(B\GG_m)\to \cI(B\GG_m)$ is simply $\nu_r \times \nu_r$. In particular this morphism is a $\bmu_r$-gerbe over a $\bmu_r$ torsor, coresponding to the gerbe factor $B\GG_m \to B\GG_m$ and the torsor factor $\GG_m\to \GG_m$.  As discussed in  \ref{SS:inertia}, forming the inertia is compatible with fiber products. We obtain that $\cI(\cD) = \cI(D) \times_{\cI(B\GG_m)}\cI(B\GG_m) \to  \cI(D)$ is canonically a $\bmu_r$-gerbe over a $\bmu_r$ torsor. 

Since $D$ is assumed Deligne--Mumford, the image of $\cI(D) \to \GG_m$ is discrete, so the torsor is trivial (though as a group scheme it is a possibly nontrivial extension). In particular every component of $\cI(\cD)$ is a  $\bmu_r$-gerbe over the image component of $\cI(D)$. By definition this is the gerbe associated to the normal bundle of $D$, namely the pullback of $\cD$. We obtain the following formula:
$$\cI(\cD) \ = \ \bmu_r \times \cI(D)\times_D\cD.$$

Of course the group scheme structure of $\cI(\cD) \to \cD$ is not a product but the extension of the group shcheme $\cI(D)\times_D\cD \to \cD$ by $\bmu_r$ corresponding to the normal bundle of $D$. Explicitly, one can look at local models on $X$ of the form $[V/G]$, where $V$ is smooth and $G$ is the stabilizer of a geometric point. We may assume that $D$ is defined by an eigenfunction $x$. Denote the character  of the action of $G$ on $x$ by $\chi:G \to\GGm$.    Then a local model of $X(\sqrt[r]{D})$ is given by $[\tilde V /\tilde G]$, where  $\tilde V = \Spec_V\cO_V[u]/(u^r-x)$, and $\tilde G = G\mathop\times\limits_{\chi, \GGm,\nu_r} \GG_m$ is the natural extension of $G$ by $\bmu_r$.

A similar description follows for rigidified inertia stack. What we will need is the following: 
\begin{lemma}\label{L:rig-in-gerbe-c}
Let $(\tilde x,\tilde g)$ be an object of a component $\cZ\subset \cI(\cD)$. Let $(x,g)$ be the image object of the corresponding component $Z\subset \cI(D)$. Denote by $(\tilde x,\tilde g)\thickslash \langle\tilde g\rangle$ and $(x,g)\thickslash \langle g\rangle$ the corresponding objects of components $\overline\cZ \subset \ocI(\cD)$ and  $\overline Z \subset \ocI(D)$. Write $|\langle\tilde g\rangle| = \frac{r}{c}|\langle g\rangle|$. Then the morphism $\overline\cZ \to \overline Z$ is a gerbe banded by $\bmu_c$.
\end{lemma}
\begin{proof} 
We have that  $\cZ \to \overline\cZ$ is a gerbe banded by $\langle\tilde g\rangle$ and $Z \to \overline Z$ is a gerbe banded by $\langle g\rangle$. Also $\cZ \to Z$ is banded by $\bmu_r$. Chasing the diagram shows that $\overline\cZ \to \overline Z$ is indeed a gerbe banded by a cyclic group, and its order is clearly the ratio $|\langle g\rangle| \cdot r / \, |\langle \tilde g\rangle| = c$. 
\end{proof}


\subsection{Twisted curves}\label{SS:tw-curves}

A {\em prestable twisted curve} with $n$ markings  is a one dimensional separated connected Deligne--Mumford stack $\cC$, with at most nodal singularities, together with  a collection of disjoint closed substacks $\Sigma_1,\ldots,\Sigma_n$ of the smooth locus of $\cC$  such that: 
\begin{enumerate}
\item the open locus $\cC^\sm\setminus \bigcup \Sigma_i$ in $\cC$ is a scheme;
\item each node is a balanced node. 
\end{enumerate}
The latter condition means that locally $\cC$ looks like the {\em model balanced node} of index $r_p$
$$N_{r_p} := \left[\,\Spec \left(\CC[u,v]/(uv)\right)\ \big/\ \bmu_{r_p}\right]$$
where the action of $\bmu_{r_p}$ is {\em balanced}, namely $(u,v) \mapsto (\zeta_{r_p} u,\zeta_{r_p}^{-1} v)$. 

Similarly, locally near each $\Sigma_i$ the twisted curve $\cC$ looks like
$$\left[\Spec \CC[u]\ /\ \bmu_{e_i}\right]$$
where $\bmu_{e_i}$ acts via $u \mapsto \zeta_{e_i} u$. The integers  $e_1,\ldots,e_n$ are the {\em orbifold indices} of the markings $\Sigma_i$.


A family of twisted prestable curves with orbifold indices $e_1,\ldots,e_n$ is a flat morphism $\cC\to S$ together with closed substacks $(\Sigma_1,\ldots,\Sigma_n)$ of $\cC$ such that:
\begin{enumerate}
\item $\Sigma_i$ is a gerbe banded by $\bmu_{e_i}$ over $S$;
\item each fiber $(\cC_s,\Sigma_{1,s},\ldots,\Sigma_{n,s})$ is a twisted prestable curve with orbifold indices $e_i$.
\end{enumerate}

\subsubsection{Automorphism groups of twisted curves}\label{SSS:aut-tw-curve}
By \cite[Lemma 4.2.3]{Abramovich-Vistoli},  see  \ref{2stacks_as_stacks},  the 2-groupoid of twisted curves is equivalent to a stack, so we can speak of automorphisms of twisted curves.

Let $\pi:\cC \to C$ be the morphism from a twisted curve $\cC$ to its coarse moduli space $C$. Since the formation of $C$ is functorial, the automorphism group of $\cC$ acts on $C$. Consider the group $\Aut_C\cC$ of automorphism of $\cC$ acting trivially on $C$. As shown in \cite[Proposition 7.1.1]{ACV} there is a canonical isomorphism
$$\Aut_C\cC \simeq \prod_{p\in \Sing(\cC)} \mu_{r_p}.$$ Notice that nodes contribute but markings do not. 
These automorphisms are known as {\em ghost automorphisms}, as they become ``invisible" when looking only at  $C$. The action of $\mu_{r_p}$ is induced on local coordinates by
$(\xi,\eta) \mapsto (\zeta_{r_p} \xi,\eta)$, equivalently $(\xi,\eta) \mapsto (\xi,\zeta_{r_p} \eta)$. We further discuss these through gluing data in \ref{SSS:bal-node-aut} below.

\subsubsection{Deformations of  twisted curves}\label{SSS:def-tw-curve}
We now consider {\em proper} twisted curves. Just like the untwisted case, deformations of twisted curves are unobstructed \cite[Proposition 2.1.1]{AJ}, \cite[3.0.3]{ACV}. The infinitesimal theory is identical to the untwisted case: first-order infinitesimal automorphisms are in the group $Hom(\Omega^1_\cC(\sum \Sigma_i), \cO_\cC)$, first-order deformations in $Ext^1$ and obstructions vanish since $Ext^2=0$.


Again as in the untwisted case, the sheaf $\cExt^1(\Omega^1_\cC(\sum \Sigma_i), \cO_\cC)$ is a sum of one-dimensional skyscraper sheaves supported at the nodes: $\cExt^1(\Omega^1_\cC(\sum \Sigma_i), \cO_\cC) = \oplus_{p\in \Sing(\cC)}\cExt^1(\Omega^1_\cC(\sum \Sigma_i), \cO_\cC)_p$. The balanced action condition guarantees that the action of the stabilizers on these skyscraper sheaves is trivial, therefore they are generated by sections. The local-to-global spectral sequence for $Ext$ gives epimorphisms
$$\Ext^1(\Omega^1_\cC(\sum \Sigma_i), \cO_\cC) \to H^0\left(\cExt^1(\Omega^1_\cC(\sum \Sigma_i), \cO_\cC)_p\right)$$
for all $p\in \Sing(\cC)$, corresponding to a divisor $\Delta_p$ in the versal deformation space - the locus where the node $p$ persists.

\subsubsection{Comparing deformations and automorphisms of $\cC$ and $C$}\label{SSS:def-tw-untw}

Since $\Aut_C\cC$ is discrete, it does not affect first-order infinitesimal automorphisms, so the homomorphism $$Hom(\Omega^1_\cC(\sum \Sigma_i), \cO_\cC)\to Hom(\Omega^1_C(\sum \bar\Sigma_i), \cO_C)$$ is an isomorphism. 

On the other hand the action of $\Aut_C\cC$ on $H^0\left(\cExt^1(\Omega^1_\cC(\sum \Sigma_i), \cO_\cC)\right)$ is easily seen to be effective. The deformation spaces of $\cC$ and $C$ are smooth and have the same dimension. It follows that the deformation space $\Def_{\cC,\Sigma_i}$ of the twisted curve is a branched cover of the deformation space $\Def_{C,\bar\Sigma_i}$ of the coarse moduli space, with index $r_p$ along $\Delta_p$. See \cite[3.5]{Cime}, \cite[1.10]{Olsson-twisted}. If we denote by $\Delta_{\bar p}$ the corresponding divisor in $\Def_{C,\bar\Sigma_i}$, then the pullback of the divisor $\Delta_{\bar p}$ in $\Def_{\cC,\Sigma_i}$ is the divisor $r_p \Delta_p$. It follows that over the smooth locus of $\Delta_p$, the scheme-theoretic inverse image of  $\Delta_{\bar p}$ is locally of the form $\Delta_p\times \Spec \CC[\epsilon]/(\epsilon^{r_p})$.

\subsubsection{Moduli of twisted curves}\label{SSS:mod-tw-curve}

Families of twisted prestable curves of genus $g$ with $n$ markings form an algebraic stack, denoted  $\fM_{g,n}^\tw$; its connected components are labeled by the indices $e_i$ of the markings, and they are all isomorphic to each other (\cite[Theorem 4.2.1]{AGV},  \cite[Theorem 1.8]{Olsson-twisted}). We have an evident embedding $\fM_{g,n}\hookrightarrow \fM_{g,n}^\tw$, since a curve is in particular a twisted curve. Taking coarse moduli spaces gives a left inverse morphisms $\fM_{g,n}^\tw\to \fM_{g,n}$. Even when fixing $e_i$, the latter morphism  is not of finite type and not separated - over the nodal locus of $\fM_{g,n}$ there are infinitely many boundary components of $\fM_{g,n}^\tw$ corresponding to different indices at the nodes. Along a boundary component $\Delta\subset \fM_{g,n}^\tw$ corresponding to a node with index $r_p$, this morphism is branched with index $r_p$ as described above. A similar result holds on universal families \cite[Section 3.3]{Cime}, given by the usual identification of the universal families in terms of moduli of curves with an additional untwisted marking \cite[Corollary 9.1.3]{Abramovich-Vistoli}.

\subsection{Maps and lifts} We now follow the key observation of Cadman's work \cite{Cadman1}, especially Theorem 3.3.6, relating tangency with orbifold maps. Suppose now $(C, \Sigma)$ a twisted curve, $(X,D)$ a locally smooth pair and $g: C \to X$ a morphism such that $g^* D  = c\Sigma$. We call $c$ the {\em contact order} of $C$ and $D$ at $p$. (Note that if $\Sigma$ is twisted with index $e$, then $\deg\Sigma =1/e$ and  the intersection multiplicity of $C$ and $D$ at $P$ is $d = c/e$.) 

\begin{lemma}\label{L:lift-pd-pair}
 Let $r,r_\Sigma\geq 1$ be  integers, let $X' = X(\sqrt[r]{D})$ with divisor $\cD$, and similarly $C' = C(\sqrt[r_\Sigma]{\Sigma})$ with divisor $\tilde\Sigma$. We assume $c$ divides $r$. Then
\begin{enumerate}
\item $g:C\to X$ lifts to $\tilde g:C'\to X'$ if and only if $r| c\cdot r_\Sigma$,
\item when $r| c\cdot r_\Sigma$ such lift is unique up to unique isomorphism, 
\item the lift is representable if and only if $r=c\cdot r_\Sigma$, and
\item the lift is transversal if and only if $r=c\cdot r_\Sigma$.
\end{enumerate}
\end{lemma}
\begin{proof} (see \cite[Theorem 3.3.6]{Cadman1})
For (1),  a lift corresponds to a line bundle $M$ on $C'$ and a section $s$ such that $s^r = \tilde g^* \bone_D$. But $g^*\bone_D = \bone_\Sigma^c$ and $\bone_\Sigma$ pulls back to $\bone_{\tilde\Sigma}^{r_\sigma}$ on $C'$. So $\tilde g^*\bone_D= \bone_{\tilde\Sigma}^{c\cdot r_\Sigma}$. It follows that if a lift exists then $r| c\cdot r_\sigma$. And if  $r d =  c\cdot r_\sigma$ for some integer $d$, the pair $(\cO(d\tilde\Sigma),\bone_{\tilde\Sigma}^d)$ gives the desired lift $C'\to X'$.

(2) Uniqueness follows since $(\cO(\tilde\Sigma),\bone_{\tilde\Sigma})|_{C\setminus \Sigma} \simeq (\cO,\bone)$ which has no nontrivial automorphisms.

(3)  It suffices to consider points over $\Sigma$. Here automorphisms of an object $(L,s=0,L^{r_\Sigma}\simeq \cO(\tilde\Sigma))$ of $C'$ are given by $L \stackrel{\zeta}{\to} L$ with $\zeta$ a primitive $r_\Sigma$ - root. The action on   $\cO(d\tilde\Sigma)= L^d$ is via $\zeta^{d}$, which has a nontrivial kernel exactly when $d>1$. 

(4) We have $\tilde g^* \cD = d\tilde\Sigma$, so $\tilde g$ is transversal if and only if $d=1$. 
\end{proof}
\begin{remark}
A similar result holds without the assumption that $c|r$, with part (3) modified. See \cite[Theorem 3.3.6]{Cadman1}.\end{remark}

Given a morphism $(C,\Sigma) \to (X,D)$, the restriction $\Sigma \to D$ is an object of a component $Z$ of $\ocI(D)$. When $r = c r_\Sigma$, the representable and transversal lift $(C',\tilde\Sigma) \to (X',\cD)$ gives rise to an object $\tilde\Sigma \to \cD$ of a component $\cZ$ of $\ocI(\cD)$ lying over $Z$. The following computation will be used in the degeneration formula:

\begin{lemma}\label{L:eval-lift-c}
The morphism $\cZ \to Z$ is a gerbe banded by $\bmu_c$.
\end{lemma} 
\begin{proof}
The object $\Sigma \to D$ is the rigidification of an object corresponding to $(x,g)$ with $|\langle g \rangle| = e$, and  $\tilde\Sigma \to \cD$ is the rigidification of an object corresponding to $(\tilde x,\tilde g)$ with $|\langle \tilde g \rangle| = e \cdot r_\Sigma.$ But $r_\Sigma = r/c$ and the result follows from Lemma \ref{L:rig-in-gerbe-c}. 
\end{proof}

\subsection{Twisting along a nodal divisor}\label{balanced_node}



We now want to generalize the singularity structure of a twisted curve  to the following case. Let $W$ be an algebraic stack, nodal and first order smoothable along a closed substack $D$, as discussed in \ref{SS:1st-order-sm}. Assume also that $W=X_1\sqcup_D X_2$ where $X_1$ and $X_2$ are closed substacks, which along $D$ are smooth and intersect transversally. The ordering of $X_i$ will be kept throughout the discussion. Assume we are given an isomorphism of $\cExt^1(\Omega_W,\cO_W)|_D$ with $\cO_D$, i.e., an isomorphism $\alpha:N_1\to N_2^\vee$, where $N_i=N_{D/X_i}$. 

\begin{defn}\label{D:twist-node}
 We define the stack obtained by {\em twisting $W$ along $D$} with index $r$, or {\em adding a balanced node structure} of index $r$ along $D$, denoted $W(\sqrt[r]{D})$, as follows.
Let $X'_i:=X_i(\sqrt[r]{D})$, and $D'_i\subset X_i'$ the reduced inverse image of $D$, so that $D_i'$ is isomorphic to $\sqrt[r]{N_i/D}$. Let $\pi_i:D_i'\to D$ be the structure morphism, and let $\phi_i:L_i^{\otimes r}\to \pi_i^*N_i$ the universal line $r$-th root. Let $\beta:D_1'\to D_2'$ be the morphism defined by the triple $(L^\vee_1,\pi_1,\alpha\circ (\phi_1^\vee)^{-1})$; it is easy to see that $\beta$ is an isomorphism. We define $W(\sqrt[r]{D})$ to be the stack obtained by gluing $X_1'$ to $X_2'$ along the identification of $D_1'$ with $D_2'$ via $\beta$. We let $D'$ be the closed substack image of either $D_i'$.
\end{defn}

\begin{remark}
A generalization of this construction and more information about it may be found in \cite{ACFW} and \cite{Borne-Vistoli}.
\end{remark}

\subsubsection{Gluing as pushout}\label{2_pushout} In Definition \ref{D:twist-node}, the term {\em gluing}  means the existence of a $2$-pushout diagram$$\xymatrix{
D' \ar[r]\ar[d] &X_1'\ar[d]\\
X_2'\ar[r] &W(\sqrt[r]{D}).
}$$ 
See \cite[Appendix A]{AGV}. In particular, for every algebraic stack $Y$, the groupoid of morphisms from $W(\sqrt[r]{D})$ to $Y$ has as objects the triples $(f_1,f_2,\eps)$ where $f_i:X_i'\to Y$ is a morphism and $\eps$ is a $2$-morphism between $f_1|_{D_1'}$ and $f_2|_{D_2'}\circ\beta$; a $2$-morphism from $(f_1,f_2,\eps)$ to $(g_1,g_2,\zeta)$ is a pair of $2$-morphisms $\alpha_i:f_i\to g_i$ such that the two $2$-morphisms from $f_1|_{D_1'}$ to $g_2|_{D_2'}\circ \beta$ induced respectively by $\eps$ and $\alpha_2$ and by $\alpha_1$ and $\zeta$ coincide.

In particular, the structure maps $X_i'\to W$ and the identity of $\beta$ define a morphism $\pi:W(\sqrt[r]{D})\to W$.

\begin{lemma} \begin{enumerate}
        \item $W(\sqrt[r]{D})$ is an algebraic stack which is nodal and first-order smoothable along $D'$;
	\item $\pi$ is proper,  quasifinite and an isomorphism over $W\setminus D$.
             \end{enumerate}
\end{lemma}
\begin{proof}
(1) the only nontrivial part to check is the first-order smoothability. This come from the identification of $L_2$ with $L_1^\vee$ via $\beta$, and the fact that $L_i=N_{D_i'/X_i'}$. 

(2) Properness is local in the smooth topology of the target, so we may assume that $W=\AA^n\times N_1$; then $W(\sqrt[r]{D})=\AA^n\times N_r$, and we are done since $N_r\to N_1$ is proper. Quasifiniteness is also smooth local in the target and follows in the same way. The isomorphism over $W\setminus D$ is obvious from the definition, as it is the case $D=\emptyset$.
\end{proof}

\subsubsection{Automorphisms}\label{SSS:bal-node-aut} Unlike root stacks, twisted nodal stacks do have non-trivial relative automorphisms.

\begin{prop}
 Let $h_i$ be an automorphism of $X_i$ restricting to the identity on $D$, and let $h:W\to W$ be the induced automorphism. Each $h_i$ acts on $N_i$ by multiplication by a nowhere vanishing regular function $\lambda_i$ on $D$. Then the set of isomorphism classes of automorphisms of $W(\sqrt[r]{D})$ lifting $h$ (in the sense of Lemma \ref{rigid_liftings}) is in natural bijection with the set of regular functions $\eps$ on $D$ such that $\eps^r=\lambda_1\lambda_2$.
\end{prop}
\begin{proof} Write $W':=W(\sqrt[r]{D})$ for brevity. 
We apply  \ref{2_pushout} with the stack $Y=W'$. Let $h':W'\to W'$ be a lifting, and $(h_1',h_2',\eps)$ the corresponding triple. Then $h_1'$ and $h_2'$ are determined up to unique $2$-isomorphism by Lemma \ref{rigid_liftings}. All that is left is to find a two-isomorphism $\eps:h_1'|_{D_1'}\to h_2'\circ \beta$ lifting the $2$-isomorphism $\lambda:h_1|_{D_1}\to h_2|_{D_1}$ induced by $h$. Since $h|_D$ acts on $L=\cExt^1(\Omega_W,\cO_W)|_D$ by multiplication with $\lambda_1\lambda_2$, we have to find an automorphism $\eps$ of its $r$-th root whose $r$th power is $\lambda_1\lambda_2$. Every automorphism of a line bundle is a nowhere vanishing function $\eps$, and the lifting condition means $\eps^r=\lambda_1\lambda_2$.
\end{proof}

\begin{corollary}  In particular the group of automorphisms of $W(\sqrt[r]{D})$ in the sense of \ref{2stacks_as_stacks} lifting the identity of $W$ is naturally isomorphic to $\bmu_r(D)$; if $D$ is connected, then it is just $\bmu_r$. \end{corollary}
As in the case of curves, elements of this group are called ghost automorphisms. 

\subsubsection{}\label{SS:bal-node-def}

  See \cite{Cime}, Section 3.3: suppose given a one-parameter smoothing of $W$ along $D$, that is a flat morphism $\pi:\cW\to B$ with $B$ a smooth curve and an isomorphism of $\cW_0:=\pi^{-1}(b_0)$ with $W$ such that $\cW$ is smooth along $D$. Then the balanced node structure can be described more directly as follows. Let $\cW'$ be the fiber product over $\tilde W$ of $\cW(\sqrt[r]{\cW_1})$ and $\cW(\sqrt[r]{\cW_2})$; by the universal property, the morphism $\cW'\to B$ induces $\tilde \pi:\cW'\to B(\sqrt[r]{b_0})$. Then $W(\sqrt[r]{D})$ is just the fiber of $\tilde \pi$ over a geometric point of $B(\sqrt[r]{b_0})$ over $b_0$.

Note that the one-parameter smoothing induces a trivialization of $\cExt^1(\Omega_W,\cO_W)$, unique up to a non-zero scalar (corresponding to a choice of the basis of $T_{b_0}B$).

\subsubsection{Twisted curves and twisted nodes}\label{SSS:tw-c-tw-n} Consider a node $\wp$ of index $r$ on a twisted curve $\cC$, and assume the two branches of $\cC$ at $\wp$ belong to two different components of $\cC$. Replacing $\cC$ by an open neighborhood of $\wp$ we may assume $\cC = \cC_1\sqcup_\wp\cC_2$. Let $C$ be the coarse moduli space of $\cC$ and $p$ the image of $\wp$. It follows from the definition that $\cC = C(\sqrt[r]{p})$.

Even if the node $\wp$ is not locally separating in the Zariski topology, there is an \'etale neighborhood where it is, so in fact every twisted node is obtained by this construction locally in the \'etale topology. See \cite{Olsson-twisted} for a formalism which works in general; we will not use this generality in this paper.

\subsubsection{Maps and lifts}

Here is the result analogous to \ref{L:lift-pd-pair} for nodes:

\begin{lemma}\label{L:lift-pd-node}
Let $C=C_1\sqcup_\Sigma C_2$ be a nodal twisted curve and let $W=X_1\sqcup_D X_2$ have first-order smoothable nodal singularities. Let $g: C \to W$ be a representable morphism, given via $g_i : C_i \to X_i \subset W$ and an isomorphism $\phi: (g_1)|_\Sigma \to (g_2)|_\Sigma$. Assume $g_i^* D = c\cdot \Sigma$ as Cartier divisor on $C_i$ (so the two contact orders agree). Let $r,r_\Sigma\geq 1$ be an integers, and consider the twisted  structures $W' = Y(\sqrt[r]{D})$ and  $C' = C(\sqrt[r_\Sigma]{D})$. We again assume $c$ divides $r$. Then
\begin{enumerate}
\item $g:C\to W$ lifts to $\tilde g:C'\to X'$ if and only if $r| c\cdot r_\Sigma$,
\item such lift is representable if and only if $r=c\cdot r_\Sigma$, and
\item such lift is transversal if and only if $r=c\cdot r_\Sigma$.
\end{enumerate}
\end{lemma}
\begin{proof} The necessity in (1) as well as (2) and (3) hold since the Lemma \ref{L:lift-pd-pair} applies to $g_i$.
To show that $r| c\cdot r_\Sigma$ is sufficient in (1),  consider the lift $\tilde g_i$ of $g_i$  corresponding to the line bundle $\cO(d\tilde\Sigma)$ on $(C_i)'$ with section $s_i = \bone_{\tilde\Sigma}^d$, where $rd=cr_\Sigma$ as in Lemma \ref{L:lift-pd-pair}; it clearly satisfies $s_i^r = \tilde g_i^* \bone_E$. 

The isomorphism $\beta$ of \ref{D:twist-node} sends $\cO_{(C_1)'}(\tilde\Sigma)|_{\tilde\Sigma}$ to   $\cO_{(C_2)'}(\tilde\Sigma)^\vee|_{\tilde\Sigma}$. In particular it sends $\cO_{(C_1)'}(d\tilde\Sigma)|_{\tilde\Sigma}$  to   $\cO_{(C_1)'}(d\tilde\Sigma)|_{\tilde\Sigma}$, giving the gluing data for a lift $\tilde g: C' \to X'$.
\end{proof}
\begin{remark}
In this case we do not have uniqueness since we may compose by ghost automorphisms.
\end{remark}

 \section{Expanded pairs and degenerations}\label{S:Expansions}

\subsection{Expanded pairs}

\begin{conv} In this subsection we fix a locally smooth pair $(X,D)$, i.e. $X$ is an algebraic stack smooth along a smooth divisor $D$. 
Write $N$ for $N_{D/X}$. Let $\PP:=\PP_D(N\oplus \cO_D)$, with its sections $D^-$ and $D^+$ having normal bundle $N^\vee$ and $N$ respectively. Let $\PP_i$ (for $i\ge 1$) be copies of $\PP$, and let $D_i^-$ and $D_i^+$ be the corresponding sections. For a positive integer $r$, let $\cD_r:=\sqrt[r]{N/D}$, see \ref{SS:root}.
\end{conv}

\begin{defn} Let $k\ge 0$ be an integer. Define $X[k]$ to be the algebraic stack obtained by gluing $X$ to $\PP_1$ via the natural isomorphism of $D$ with $D_1^-$, and for every $i\in\{1,\ldots,k-1\}$ gluing $\PP_i$ with $\PP_{i+1}$ via the natural isomorphism of $D_i^+$ with $D_{i+1}^-$.

 \begin{equation*}
X[k] := X_1 \hskip -1ex \mathop{\sqcup}\limits_{\hphantom{_-}D=D_1^-}  \PP_1 \mathop{\sqcup}\limits_{D_1^+=D_2^-}  \cdots   \mathop{\sqcup}\limits_{D_{k-1}^+=D_k^-}  \PP_k
\end{equation*}

Note that $X[k]$ has nodal first-order smoothable singularities along its singular locus, which is the disjoint union of $D_0,\ldots,D_{k-1}$ (where $D_i$ is the image of either $D_{i+1}^-$ or $D_{i}^+$). Write $D_k$ for the image of $D_k^+$; note that $X[k]$ is smooth along $D_k$ and that $\iota^*N_{D_k/X[k]}=N$, if $\iota:D\to D_k\subset X[k]$ is the natural isomorphism.

We call the sequence of morphisms $D\to X[k]\to X$ an 
{\em untwisted expanded pair}. (It is also sometimes called an   {\em untwisted half accordion}.)
of length $k$.\end{defn} 

\begin{defn}
Let ${\bf r}:=(r_0,\ldots,r_{k})$ be a sequence of positive integers. Define $X[k]({\bf r})\to X[k]$ to be the morphism obtained by twisting $X[k]$ along $D_i$ with index $r_i$; note that we can apply the construction in section \ref{balanced_node} since the relevant normal bundles are naturally dual to each other. 
We call the sequence of morphisms $\cD_{r_k}\stackrel{\iota}{\to} X[k]({\bf r})\to X$  the {\em ${\bf r}$-twisted expanded pair of length $k$} over $(X,D)$, or just a twisted {\em expansion} of (X,D). Its {\em twisting index}
is defined to be the integer $r_k$. 
\end{defn} 

Note that an untwisted expansions of length $k$ is the same as a ${\bf 1}$-twisted expansion of length $k$, where ${\bf 1}=(1,\ldots,1)$. Moreover, there is a natural morphism $X[k]({\bf r})\to X[k]$ which, if $X$ is a scheme or algebraic space, is just the morphism to the coarse moduli space. 

\subsubsection{Stacks of expanded pairs}\label{Sec:pairs} Forming stacks of \ChDan{families of} expanded pairs is slightly more subtle than it might seem, since in general not all deformations are good.
The stack $\cT^u$ of untwisted expanded pairs was defined by Jun Li, see \cite[Proposition 4.5]{Li1} where the notation ${\mathfrak Z}^{rel}$ is used for its universal family, and \cite[Section 2.8]{Graber-Vakil} where the notation $\cT$ for the stack we denote here by $\cT^u$ is introduced. In this paper we also use the stack $\cT^\tw$ of twisted expanded pairs. 
The stacks $\cT^u$ and $\cT^\tw$ are studied in \cite{ACFW}, where various alternative definitions and properties are described. It is shown in \cite[Theorem 1.3.2]{ACFW} that the various definitions of $\cT^u$ coincide. We rely on \cite[Definitions 2.1.5 and 2.4.2(1)]{ACFW} in this text. 

We denote by $\cT_r^\tw \subset \cT^\tw$ the open and closed substack of twisted expansions of index $r$, so that $\cT^\tw = \sqcup \cT_r^\tw$.

We denote the universal families of expanded pairs $\cX^u_{univ} \to \cT^u$ and   $\cX^\tw_{univ} \to \cT^\tw$.

\begin{proposition}\label{Prop:cT}
\begin{enumerate}
 \item The stacks $\cT_r^\tw$ and $\cT^u$ are smooth  connected algebraic stacks of dimension 0 (\cite[Theorem 1.3.1]{ACFW}).
\item The stacks $\cT^\tw$ and $\cT^u$ do not depend on the choice of the pair $(X,D)$ (\cite[Remark 2.1.7]{ACFW}).
\end{enumerate}
\end{proposition}


\subsection{Twisted expanded degenerations}\label{expdeg}
We now discuss stacks analogous to the Artin stack underlying Jun Li's family $\fW$ of expanded degenerations.

\begin{conv}\label{degeneration_conv}
In this subsection we fix $\pi: W \to B$, a flat morphism such that $B$ is a smooth curve, $W$ is a smooth stack, and $b_0\in B$ is the unique critical value of $\pi$; we set $W_0:=\pi^{-1}(b_0)$ an assume $W_0= X_1\sqcup_D X_2$  is the union of two smooth closed substacks $X_1$ and $X_2$ intersecting transversally along $D$, a smooth divisor in each $X_i$. This implies that $W_0$ is nodal and first-order smoothable along its singular locus $D$. 
\end{conv}

\begin{defn} 
\begin{enumerate}
\item Let $k\ge 0$ be an integer. Define $W_0[k]$ to be the Deligne--Mumford stack obtained by gluing the $k$-th untwisted \ChDan{expanded pair}  $X_1[k]$ over $(X_1,D)$ to $X_2$ via the identification of $D_k\subset X_1[k]$ with $D\subset X_2$; denote the image of $D_i\subset X_1[k]$ again by $D_i$.  
\item Let ${\bf r}:=(r_0,\ldots,r_k)$ be a sequence of positive integers. Define $W_0[k]({\bf r})\to W_0[k]$ to be the morphism obtained by twisting $W_0[k]$ along $D_i$ with index $r_i$. 
\item We call the morphism $W_0[k]({\bf r})\to W_0$ the $(k,{\bf r})$ {\em twisted \ChDan{expansion}} or  over $W_0$, or just an expansion of $W_0$. We call it an {\em untwisted  \ChDan{expansion}} if ${\bf r}=(1,\ldots,1)$ or equivalently if it has the form $W_0[k]\to W_0$. An {\em expanded degeneration} of $W$ is either a general fiber or a twisted or untwisted  \ChDan{expansion}.
\end{enumerate}
\end{defn} 

Again we rely on \cite[Definitions 2.3.6 and 2.4.2(2)]{ACFW} for the definition of the stacks $\fT_{B,b_0}^u$ and  $\fT_{B,b_0}^\tw$ of untwisted and twisted expanded degenerations.


We denote the universal families of expanded degenerations $\cW^u_{univ} \to \fT^u$ and   $\cW^\tw_{univ} \to \fT^\tw$.

Here is the analogue of Proposition \ref{Prop:cT}:

\begin{proposition}\label{P:fTtw-fTu}
 \begin{enumerate}
\item The stacks $\fT^\tw_{B,b_0}$ and $\fT^u_{B,b_0}$ only depend on the base $(B,b_0)$ but not on the family $W$ (\cite[Remark 2.3.8]{ACFW}).
 \item The stacks $\fT^\tw_{B,b_0}$ and $\fT^u_{B,b_0}$ are smooth, algebraic connected stacks of dimension $1$ (\cite[Theorem 1.3.1]{ACFW}).
\end{enumerate}
\end{proposition}

\subsection{The singular case in the absence of smoothing}
We have defined $\fT^u_0 = \fT^u_{B,b_0}\times_B \{b_0\}$ and $\fT^\tw_0 = \fT^\tw_{B,b_0}\times_B \{b_0\}$. When $\pi:W\to B$ is as in \ref{expdeg} above, these stacks parametrize untwisted and twisted expansions of the singular fiber $W_0$. But more is true: consider a first-order smoothable $W_0 = X_1\sqcup_D X_2$. We define untwisted and twisted \ChDan{expansions} on $W_0$ as before, and they obviously coincide with the previous definition if $W_0$ is provided with a smoothing over some curve $B$. 
\ChDan{The general setup in \cite[Definitions 2.3.6 and 2.4.2(2)]{ACFW}  includes families of  expansions of $W_0$.} We have the following.
\begin{proposition}\begin{enumerate}
\item The stacks $\fT^\tw_{0}$ and $\fT^u_{0}$ are the stacks of twisted and untwisted expansions of  $W_0$.
\item These stacks $\fT^\tw_{0}$ and $\fT^u_{0}$ are independent of $W_0$.
\end{enumerate}
\end{proposition}
In particular,  if $W_0$ is a fiber in a one-parameter family $\pi:W \to B$, the universal families over $\fT^\tw_{0}$ and $\fT^u_{0}$ are independent of the smoothing.


Note that the  stacks $\fT^\tw_{0}$ and $\fT^u_{0}$ are singular, and $\fT^\tw_0$ is nonreduced. 



\subsection{Split expansions}\label{SS:spl-exp} As the  stacks $\fT^\tw_{0}$ and $\fT^u_{0}$ are normal crossings, their normalizations are of interest.

In \cite[Section 7.4]{ACFW} we define the stack $\fT^{r,\spl}_{0}$ to be the stack of twisted  \ChDan{expansions}  with a choice of a splitting divisor $\cD$ of twisting index $r$. We similarly  $\fT^{u,\spl}_{0}$ in the untwisted case. We have a natural map $\fT^{r,\spl}_{0}\to \fT^{r}_{0}\subset \fT^{\tw}_{0}$ and  $\fT^{u,\spl}_{0}\to \fT^{u}_{0}$.  

Consider now the stack $\fQ = \fT^{u,\spl}_{0} \times_{\fT^{u}_{0}} \fT^{\tw}_{0}$. It decomposes as a disjoint union $\fQ = \sqcup_r \fQ_r$, where over he reduction of $\fQ_r$ the splitting divisor in the universal family is twisted with index $r$. 

\begin{lemma}[{\cite[Proposition 7.4.2(2)]{ACFW}}]\label{L:tw-1/r}
The morphism $\fT^{r,\spl}_{0} \to \fQ_r$ is of degree $1/r$.
\end{lemma}

Consider the universal family $\cW^{r,\spl}_{univ}$ over $\fT^{r,\spl}_{0}$. Taking the partial normalization along the splitting divisor $\cD$ gives two families  $\cD \to \cX_1$ and $\cD \to \cX_2$ of expanded pairs, which define a morphism $\fT^{r,\spl}_{0} \to \cT^\tw_r \times \cT^\tw_r$. 

\begin{lemma}[{\cite[Proposition 7.4.2(1)]{ACFW}}]\label{L:gl-1/r}
The natural morphism $$\fT_{0}^{r,\spl} \to \cT^\tw_r \times \cT^\tw_r$$
corresponding to the two components of the partial normalization of the universal family $\cW^{r,\spl}_{univ}$ is a gerbe banded by $\bmu_r$; in particular it has degree $1/r$.
\end{lemma}

\subsection{Weighted and $\fr$-twisted expansions}\label{weighted-expansions}
 In Theorem \ref{gw-ind-twisting} we will need a slightly refined version of the stacks $\cT^\tw$ of expanded pairs and $\fT^\tw$ of expanded degenerations, namely the stacks $\cT^\fr$ and $\fT^\fr$ of $\fr$-twisted $\bDelta$-wighted expansions. 
 
\subsubsection{$\bDelta$-weighted expansions} As a first step we consider  a set $\bDelta$ and stacks $\cT^\bDelta$ and $\fT^\bDelta$ in which the splitting divisors of the universal object are weighted, in a sense analogous to Costello \cite[Section 2]{Costello}, by elements of $\bDelta$. In our application of Theorem \ref{gw-ind-twisting} the set  $\bDelta$ is the set of finite multi-sets $\bc = \{c_1,\ldots,c_n\}$ of positive integers, which will indicate contact orders of branches of a curve along a splitting divisor. (Recall that a multi-set is an unordered sequence, with elements possibly appearing several times). Unlike Costello's situation we will not need a monoid structure on the weight set. 

The construction, algebraicity and basic properties of the stacks $\cT^\bDelta$ and $\fT^\bDelta$ is detailed in \cite[Section 7.1]{ACFW}. 
 
 There are evident forgetful maps $\fT^\bDelta\to\fT^\tw$ and  $\cT^\bDelta\to\cT^\tw$. By construction these maps are \'etale and representable. The universal families on $\fT^\bDelta$ and $\cT^\bDelta$ are given by pullback along the respective forgetful map. As before we define $\fT_0^\bDelta$ to be the inverse image of $b_0$.
  
\subsubsection{$\fr$-twisted expansions} Next we consider a function $\fr: \bDelta \to \NN$ with positive integer values. An object of $\fT^\bDelta$ or $\cT^\bDelta$ is said to be {\em $\fr$-twisted} if  the  $\ell$-th splitting divisor is twisted with index $\fr(\bc^\ell)$. We obtain open substacks $\cT^\fr\subset \cT^\bDelta$ and $\fT^\fr\subset \fT^\bDelta$ of $\fr$-twisted, $\bDelta$-weighted expansions. Again $\fT_0^\fr$ denotes the inverse image of $b_0$. \ChDan{These stacks are described using logarithmic structures in \cite[Definition 7.2.3, Lemma 7.2.4]{ACFW}.}

We define a partial ordering $\prec$ on functions $\fr: \bDelta \to \NN$ by divisibility: $\fr\prec\fr'$ if an only if $\fr(\bc)|\fr'(\bc)$ for all $\bc\in \bDelta$.

\begin{lemma}[{\cite[Section 7.3]{ACFW}}]\label{compare-twisted}
Assume that we are given functions  $\fr$ and $\fr':\bDelta\to \NN$ such that $\fr\prec\fr'$. Then are natural partial untwisting morphisms $\fT_{B,b_0}^{\fr'}\to\fT_{B,b_0}^\fr$, $\fT_0^{\fr'}\to\fT_0^\fr$ and $\cT^{\fr'}\to\cT^\fr$. These morphisms lift canonically to universal families, i.e.: for every algebraic  stack $S$ and every family of expansions $\cW'\to S$ corresponding to a morphism $f':S\to \fT^{\fr'}$, let $\cW\to S$ be the family induced by the composite morphism $f:S\to \fT^\fr$; then there is a natural morphism $p:\cW'\to \cW$ lifting $\fT_{B,b_0}^{\fr'}\to\fT_{B,b_0}^\fr$.
\end{lemma}


The following  proposition is a manifestation of the well-known fact that given a twisted curve, there is essentially a unique way to increase its indices by any given amount.

\begin{proposition}\label{P:untwisting-proper}
The morphisms $\fT_{B,b_0}^{\fr'}\to\fT_{B,b_0}^\fr$, $\fT_0^{\fr'}\to\fT_0^\fr$ and $\cT^{\fr'}\to\cT^\fr$ are proper, quasi-finite, flat and surjective. Moreover this map has pure degree 1 in the sense of \cite[Section 5]{Costello}.
\end{proposition}
\begin{proof}
We indicate the proof for  $\cT^{\fr'}\to\cT^\fr$, the other cases are identical. This is essentially \cite[Proposition 7.3.1]{ACFW}, except that the fact that the morphisms are proper, quasi-finite, flat and surjective was not stated there. We use the proof of \cite[Proposition 7.2.2]{ACFW} instead:  the required properties are local on an $S$-point of $\cT^\fr$ and we may assume the logarithmic structure on $S$ is free. Then  $\cT^{\fr'}\times_{\cT^{\fr}} S \to S$ is a root stack, which is indeed  proper, quasi-finite, flat and surjective. 
\end{proof}

\newcommand{\TTT}{\mathfrak T}

 \section{Stacks of twisted stable maps and their properness}\label{S:Stablemaps}

\subsection{Conventions on stacks of maps}\label{S:Stablemaps-conventions}

We will extend Jun Li's construction \cite[Section 3]{Li1} to the case where we degenerate an orbifold instead of a smooth variety, and we introduce the twisted version of that stack which allows us to replace predeformable maps by transversal ones. We do the same for pairs as well. In this sections, the ambient space for both degenerations and pairs will be denoted by the letter $W$.

\ChDan{By a {\em curve class} on $W$ we mean an algebraic equivalence class $\beta = [\underline C]$ of an algebraic curve $\underline C\subset \underline W$ in the coarse moduli space of $W$, an element of the Chow group modulo algebraic equivalence. } 

We will use data as set in one of the following two conventions:

\begin{conv}[Data for a degeneration] \label{C:deg-data} Consider a morphism $\pi:W\to B$ and $W_0=X_1\sqcup_DX_2$ as in Section \ref{expdeg}
We also fix  $\Gamma=(\beta,g,N,\bfe)$ where \begin{enumerate}
\item $\beta$  is a curve class in the fiber of $W\to B$;
\item  $g\ge 0$ is an integer;
\item $N$ is a finite ordered set, possibly empty, which we may take to be $\{1,\ldots,n\}$.
\item ${\bfe}=(e_i)_{i\in N}$ is a tuple of positive integers such that $\cI_{e_i}(W_b)\ne \emptyset$ for all $b\in B$, $i\in N$; note that if $W$ is a variety, one must have $e_i=1$ for all $i$.
\end{enumerate}
\end{conv}

\begin{conv}[Data for a pair] \label{C:pair-data} Fix a smooth pair $(W,D)$ with $W$ a  Deligne--Mumford stack. We also fix
$\Gamma=\{\beta,g, N,M,\bfe,\bff, \bfc\}$ where\begin{enumerate}
\item $\beta$ is a curve class on \ChDan{$W$};
\item  $g\ge 0$ is an integer;
\item $N,M$ are disjoint finite ordered sets, which we may take to be $\{1,\ldots,n\}$ and $\{n+1,\ldots,n+|M|\}$.
\item ${\bfe}=(e_i)_{i\in N}$ as above, and ${\bff}=(f_j)_{j\in M}$ is similarly a tuple of positive integers such that $\cI_{f_j}(D)\ne \emptyset$ for all  $j\in M$.
\item ${\bfc}=(c_j)_{j\in M}$ are positive integers such that 
 $\sum_{j\in M} c_j/f_j=(\beta\cdot D)_X$. 
\end{enumerate}
We denote by ${\bfd}=(d_j)_{j\in M}$ the tuple formed by $d_j = c_j/f_j$.
\end{conv}

\begin{remark} Following Jun Li, we will later find it useful to think in either case of the data $\Gamma$ as a weighted modular graph. In the degenerate case it has one vertex marked with $(g,\beta)$, no edges or loops, legs corresponding to the set $N$ and weighted by $e_i$, and no roots. In the pair case it has again one vertex, no edges or loops, legs corresponding to the set $N$ and weighted by $e_i$, and roots corresponding to the set $M$ and weighted by $(f_j,c_j)$. See definition \ref{D:modular-graph}.
Indeed, the degeneration formula requires working with disconnected graphs. To avoid heavy notation here  we postpone the disconnected case to section \ref{SS:disconnected-maps}.
\end{remark}


Consider the universal families $\cW^\tw_{univ}\to \fT^\tw_{B,b_0}$ and $\cW^u_{univ}\to \fT^u_{B,b_0}$ of twisted and untwisted expanded degenerations. They both satisfy the assumptions in Convention \ref{flsp} with $D=\emptyset$. Similarly consider the universal families of pairs $(\cW_{univ}^\tw,\cD_{univ}^\tw) \to \cT^\tw$ and $(\cW_{univ}^u,\cD_{univ}^u) \to \cT^u$.  These  satisfy the assumptions in   \ref{flsp} with $D = \cD_{univ}$. 

\begin{conv}[Shorthand for stacks of maps] The notation $\fK^\tw$, $\fK^u$, and $K(W)$ will be used  for  either one of the following three cases 
\begin{enumerate} 
\item (Degeneration case) $W \to B$ is a degeneration as in \ref{C:deg-data}, $T^\tw = \fT^\tw_{B,b_0}$,  $T^u = \fT^u_{B,b_0}$.
$$\fK^\tw:= \fK_{\Gamma}(\cW^\tw_{univ}/\fT^\tw_{B,b_0}),  \quad 
    \fK^u := \fK_{\Gamma}(\cW^u_{univ}   /\fT^u_{B,b_0}) ),$$ $$ K(W) := \fK_{\Gamma}(W/B).  $$ 
\item (Singular case) $W=W_0$ is first-order smoothable, $B = \Spec \CC$, $T^\tw = \fT^\tw_0$,  $T^u = \fT^u_0$.  
$$\fK^\tw:= \fK_{\Gamma}(\cW^\tw_{0\,univ}/\fT^\tw_0),  \quad 
    \fK^u := \fK_{\Gamma}(\cW^u_{0,univ}   /\fT^u_0) ),$$ $$ K(W) := \fK_{\Gamma}(W_0).  $$ 
\item  (Relative case)   $(W,D)$ a pair, $T^\tw = \cT^\tw$,  $T^u = \cT^u$.  
$$\fK^\tw:=  \fK_{\Gamma}((\cW^\tw_{univ},\cD^\tw_{univ})/\cT^\tw), \quad    \fK^u :=  \fK_{\Gamma}((\cW^u_{univ},\cD^u_{univ})/\cT^u) ), $$ $$  K(W) := \fK_{\Gamma}(W,D).  $$
\end{enumerate}
  We follow  the notation  $K,\fK_{nd}$ and $K_{nd}$ of \ref{N:fK}, adding a superscript $\tw$ or $u$ to denote the corresponding substacks of $\fK^\tw$ or $\fK^u$. We will suppress the superscript $\tw$ or $u$ when statements hold for wither one.
  
 In either of the three cases we will write  $f:(C,\Sigma)\to \cW\to W$ for a stable map belonging to $\fK$ (i.e. $\fK^\tw$ or $\fK^u$). We will indicate the divisor $D$ only when necessary. 
\end{conv}

\ChDan{
\begin{remark}
We note that if we view a pair as a subvariety of a singular fiber, and view a fiber as a subvariety of a degeneration, the notion of curve class changes, as inequivalent classes can become equivalent through each transition.  
\end{remark}
}


Following \cite{Li1}, we can characterize points of $\fK$ belonging to the Deligne--Mumford locus $K$ using semistable components:

\begin{definition} Let $f:(C,\Sigma)\to \cW\to W$  be a stable map corresponding to a point in   $\fK^\tw$. Let $E\subset C$ be an irreducible component. 

We say that $E$ is {\em a semistable component} if it is a standard cyclic cover of a line occuring as a fiber in some $\PP$, explicitly  if 
\begin{enumerate} 
\item $E$ is smooth, irreducible, of genus zero, contains no marked points;
\item $E$ maps to a fiber $F$ of $\cW\to W$;
\item $f|_E:E\to F$ is branched at most over the intersection of $F$ with the singular locus of $\cW$;
\item $E\cap C^{sing}$ maps to $\cW^{sing}$.
\end{enumerate}
\end{definition}

\begin{lemma}\label{pointsincK} A point of $\fK$ given by a stable map $f:(C,\Sigma)\to \cW\to W$ is in $K$ if and only if there is no irreducible component $\PP_i$  such that every component of $C$ whose image meets  $\PP_i\setminus (D_i^+\cup D_i^-)$  is a semistable component. 
\end{lemma}
\begin{proof}
The point is in $K$ if and only if no positive dimensional subgroup of the group of automorphisms of the  \ChDan{expansion}  $\cW$ lifts to an automorphism of $C$. Up to the finitely many ghost automorphisms, such automorphism are given by a copy of $\GG_m$ for every component $\PP_i$.  The Deligne--Mumford condition is equivalent to ensuring that for each $i$ there is at least one component $X_i$ mapping to $\PP_i$ to which the $\GG_m$ action doesn't lift. It is easy to see that the only components whose image meets $\PP_i\setminus (D_i^+\cup D_i^-)$ to which the action lifts are exactly the semistable components.
\end{proof}

\subsection{Transversal maps and predeformable maps}\label{Sec:transversal-and-predeformable}

Recall the natural morphisms between $\fT^\tw_{B,b_0}$ and $\fT^u_{B,b_0}$ in Proposition \ref{P:fTtw-fTu} (3),
and the corresponding ones for the singular and relative cases. By the functoriality of \cite[Corollary 9.1.2]{Abramovich-Vistoli}, these induce an open embedding $\fK^u\to \fK^\tw$ and an untwisting morphism $\fK^\tw\to\fK^u$ which is its left inverse. These are compatible with the morphisms $\fK^u\to K(W), \fK^\tw\to K(W)$.

\begin{definition} Let $f:(C,\Sigma)\to \cW\to W$ be a stable map corresponding to a point in   $\fK^\tw$. We say that $f$ is {\em transversal} if it is transversal to both the singular locus and the boundary divisor in the sense of \ref{logsmpair}. 
\end{definition}
Note that the condition is vacuous when $W$ is nonsingular and $D$ empty.

\begin{remark} If $f$ is a transversal stable map, then it is nondegenerate.\end{remark}

By Lemma \ref{L:tr-open} the transversal condition is open. This allows us to formulate the following:

\begin{definition} We define $\fK^{tr}\subset \fK^\tw$ to be the open substack of transversal maps; we write $K^{tr}:=K^\tw\cap\fK^{tr}$.\end{definition}

The main objects we will use are transversal maps. However we find it appropriate to relate them to {\em predeformable} maps, a notion used by Jun Li and other previous authors. Accordingly, our arguments will go by way of predeformable maps, even though this detour can be avoided.

\begin{definition}\label{D:predeformable} We say that a non-degenerate map $\cC \to \cW$ is {\em predeformable} if it is in the set-theoretic image of $\fK^{tr}$ in $\fK^u$.
\end{definition}

\subsubsection{Contact orders and Li's predeformability} In \cite[Definition 2.5]{Li1}, a nondegenerate morphism $f: C \to W$ over an algebraically closed field  is defined to be predeformable if at every point $p\in C$ mapping to a singular component $D^\ell$ of $W$, locally the map is described as 
$$ \left(\Spec k[u,v]/(uv)\right)^{\sh} \lrar \left(\Spec k[x,y,z_1,\ldots,z_m]/(xy)\right)^{\sh} $$ 
where
$$ x \mapsto u^{c_p}, \quad y \mapsto v^{c_p}.
$$
for some positive integer $c_p$, called the {\em contact order at $p$}. Note that we put no condition at points mapping to the boundary divisor, if nonempty; however for points $p$ mapping to the boundary divisor we  define a contact order  as usual: locally the map is described as 
$$ \left(\Spec k[u]\right)^{\sh} \lrar \left(\Spec k[x,z_1,\ldots,z_m]\right)^{\sh} $$ where
$ x \mapsto u^{c_p}$, and the contact order is again $c_p$. 

We now show that the two notions of predeformability coincide in a precise way. The key is Lemmas \ref{L:lift-pd-pair} and \ref{L:lift-pd-node}.

\begin{lemma}\label{pd-is-pd}  \begin{enumerate}
\item Let $C\to W$ be a predeformable map in the sense of Definition \ref{D:predeformable}.
Then $C \to W$  is predeformable in the sense of \cite[Definition 2.5]{Li1}.
\item Let $f: C \to W$ be predeformable in the sense of \cite[Definition 2.5]{Li1}. For  any component $D$ of the singular locus of $W$ fix a positive integer $r_D$  divisible by the contact order $c_p$ for every node of $C$ mapping to the given singular component  $D$. Let $\cW$ be the root stack of $W$ with index $r_D$ over each component $D$ of the singular locus. Then there is a transversal map $\tilde f: \cC \to \cW$ mapping to   $f: C \to W$; in particular $f$ is predeformable in the sense of Definition \ref{D:predeformable}.
\item Any twisted stable map $\tilde f: \cC \to \cW$, with $\cW$ as in (2) above, and lifting   $f: C \to W$ is transversal.
\end{enumerate}
\end{lemma}

\begin{proof}
\begin{enumerate}
\item 
Let $\cC \to \cW$ be a transversal map which lifts $C \to W$. Consider a point $p\in C$, and a lifting $\tilde p$ in $\cC$. 
In local coordinates, if $p$ is a node the map $\cC \to \cW$ is given by $x\mapsto u, y\mapsto v$ and arbitrary $z_i\mapsto f_i(u,v)$. The coordinates on $C$ are $\bar u = u^{r_p}$ and $\bar v = v^{r_p}$. The coordinates on $W$ are $\bar x = x^{r_D}$, $\bar y = y^{r_D}$, and $z_i$ unchanged. Then  $\bar x (u^{r_p}) = u^{r_D}$ and similarly for $\bar y$.  
We thus have $r_D = c_p\cdot r_p$ for some $c_p$ and may take $\bar u = \bar x^{c_p},\ \bar v = \bar y^{c_p}$ as needed. When $p$ maps  to  boundary divisor the calculation is similar.
\item 
Applying  Lemmas \ref{L:lift-pd-pair} and  \ref{L:lift-pd-node} a transversal representable lift exists with $\cC$ having twisting index $r_p=r_D/c_p$, locally at each $p$ mapping to a component $D$. Since such a lift is unique away from the union of the $D$, these local lifts glue to a lift $\cC \to \cW$. 
\item This follows from part (4) of Lemma \ref{L:lift-pd-pair}  and part (3) of Lemma  \ref{L:lift-pd-node}.
\end{enumerate}
\end{proof}

We now consider the scheme structure for stacks of transversal and  predeformable maps. Since there are infinitely many choices for the twisting index $r_D$ in Lemma \ref{pd-is-pd}, we need the following for constructibility: 
\begin{lemma}\label{L:trans-prop}
\begin{enumerate}
\item (Base change property.) Let $\cC \to \cW \to S$ be a family of transversal maps over $S$, and let $\tilde \cW \to \cW$ be a family of root stacks. Let $\tilde \cC =\tilde \cW \times_\cW \cC$. Then $\tilde \cC\to S$ is a twisted nodal curve and $\tilde \cC \to \tilde\cW\to S$ is a family of transversal maps.
\item (Descent property.) Assume $\tilde \cC \to \tilde\cW\to S$ is a family of transversal maps with underlying untwisted maps $C\to W \to S$. Assume that there is an integer $d$, and $D\subset W$, either a component of the singular locus or the boundary divisor, such that $\tilde\cW$ is twisted  at $D$ with twisting index $r_D$ such that for any point $p$ of $C$ mapping to $D^\ell$, we have $d\cdot c_p | r_D.$ Write $r=\frac{r_{D}}{d}$ and let $\cW\to W$ be isomorphic to $W(\sqrt[r]{D})$ near $D$ and to $\tilde \cW$ elsewhere. Then the representable map $\cC \to \cW \to S$ obtained by stabilizing $\tilde \cC \to \cW\to S$ is transversal.
\end{enumerate}
\end{lemma}
\begin{proof} The transversality property is tested on fibers. 
The base change property is now a local computation, e.g. in case of a node
$$ \left(\Spec k[u,v]/(uv)\right)^{\sh} \lrar \left(\Spec k[x,y,z_1,\ldots,z_m]/(xy)\right)^{\sh} $$
where
$ x \mapsto u, \quad y \mapsto v 
$ and $z_ \mapsto f_i(u,v)$, the map $\tilde \cC \to \tilde\cW$ is given by 
$$ \left(\Spec k[\tilde u,\tilde v]/(\tilde u\tilde v)\right)^{\sh} \lrar \left(\Spec k[\tilde x,\tilde y,z_1,\ldots,z_m]/(\tilde x\tilde y)\right)^{\sh} $$
where
$ \tilde x \mapsto \tilde u, \quad \tilde y \mapsto \tilde v 
$ and $z_ \mapsto f_i(u,v)$ as before. 

For the descent property, note that $\cC$ is obtained as the relative coarse moduli space of $\tilde \cC \to \cW$. Its formation commutes with base change so we can restrict to fibers again. The integer $r$ is divisible by the contact orders, so Lemmas \ref{L:lift-pd-pair} and \ref{L:lift-pd-node} apply. In particular, since the map is representable we have $r=f_pc_p$, and therefore it is transversal. 
\end{proof}

\begin{lemma}\label{predefisclosed} The  collection of predeformable maps 
 is closed in $\fK^u_{nd}\subset \fK^u$, the open locus of nondegenerate maps.
\end{lemma}
\begin{proof} This is a local statement in the \'etale topology; hence we can use the proof given by Jun Li, see \cite[Lemma 2.7]{Li1}.
\end{proof}

Given a deformation of a nondegenerate map over a base of finite type, only finitely many contact orders occur. Lemma \ref{L:trans-prop} implies that the following gives a well defined closed substack:

\begin{definition} We define the stack of {\em predeformable maps} $\fK^\tw_{pd}$  to be the stack-theoretic image of $\fK^{tr}$ in $\fK^u$. We define $\fK^u_{pd}$, $K^u_{pd}$, $K^\tw_{pd}$ as the intersections of the appropriate sunbstacks of $\fK^\tw$ with $\fK^\tw_{pd}$.
\end{definition}


\begin{remark} A more precise form of the statement in Lemma \ref{pd-is-pd}, which we do not use in this paper, is the following: if $W\to B$ is a projective morphism of schemes, then the stack $K^u_{pd}$ is naturally isomorphic to $\fM(W,\Gamma)$ in \cite{Li1}.

 Set theoretically this is shown in Lemma \ref{pd-is-pd}. The subtle scheme structure in J. Li's stack relies on \cite[Lemmas 2.3]{Li1}, which describes the scheme structure of a predeformable map over a base scheme, and \cite[Lemmas 2.4]{Li1} which shows that the scheme structure behaves well under base change and can be glued. One can lift these to the root stacks described in Lemma \ref{pd-is-pd} and show that indeed a family of predeformable maps underlies a family of transversal maps.
\end{remark}

\begin{remark} As remarked by Jun Li, $\fK^u_{pd}$ is locally closed, but in general not open, in $\fK^u$; this makes it hard  to write down a perfect obstruction theory for it, see \cite[1.2-1.3 p. 213-129 and Appendix, p. 284-288]{Li2}. 
\end{remark}

\subsection{Properness of the stack of predeformable maps}\label{Sec:proper} Our goal below is to prove properness for certain stacks of transversal maps, which can be done directly. However we find it appropriate to relate this to previous work and go through properness of predeformable maps.
 
\begin{theorem}\label{propermappd} The natural morphism $K^u_{pd}\to K(W)$ is proper. 
\end{theorem}
 The reader who is familiar with both Jun Li's proof and the definition of twisted stable maps will be able to directly modify Jun Li's proof of \cite[Lemmas 3.8, 3.9]{Li1} to cover the orbifold case treated here. We have provided a different proof, based on stable expanded configurations of points,  in \cite{AF-config}.

\begin{corollary}\label{C:pd-proper} Under the assumptions for this section, assume moreover that $W\to B$ is proper and has projective coarse moduli space. Then $K^u_{pd}$ is proper over $B$.
\end{corollary}
\begin{proof}
This follows since by \cite[Theorem 1.4.1]{Abramovich-Vistoli} the stack $K(W)$ is proper.\end{proof}

\subsection{$\fr$--twisted stable maps and their properness}


\begin{definition}\label{D:twistingchoice} A {\em twisting choice} is a map $\fr$ which associates to every finite multi-set  of positive integers $\bc=\{c_1,\ldots,c_n\}$ a positive integer $\fr(\bc)$ such that $c_j | \fr(\bc)$ for all $j$.\end{definition}   
\begin{definition}\label{comptc}
We define a partial order  on twisting choices by saying that $\fr\prec\fr'$ if $\fr(\bc)$ divides $\fr'(\bc)$ for every $\bc$.
\end{definition}
Note that there is a unique minimal twisting choice, namely $$\fr_{\min}(\bc) = \lcm(\bc_1,\ldots,\bc_k).$$  Similarly, given any two twisting choices we obtain a third larger than both by laking their least common multiple.
  
\begin{remark} We find it important to allow choosing a twisting choice, for two reasons.  First, a non-minimal choice is used in \cite{ACW}). Second,  keeping track of the fact that the invariants we define do not depend on the twisting choice helps make sure that we are doing things right.
\end{remark}
\begin{definition} Let $\fr$ be a twisting choice. A  map   $f:(C,\Sigma)\to \cW$ is called an {\em $\fr$-twisted stable map} if it is in $K^{tr}$ and the following holds. Let ${f^u}:{C^u}\to \cW^u$ be the image of $f$ in $K^u_{pd}$. Consider any splitting divisor $D_\ell$ of $\cW$, and let  $\bc^\ell$ be the multiset of contact orders of the nodes in $C^u$ mapping to $D^u_\ell$. Then the twisting index  of   $D_\ell$ is $\fr(\bc^\ell)$. 
\end{definition}

Finally we arrive at the main moduli stacks of this paper:

\begin{definition}\label{D:rtsm} If $\fr$ is a twisting choice, we define the substack $\cK^\fr$ of $K^{tr}$ to be  the full substack of families whose points are $\fr$-twisted stable maps. In the following sections we will use case-specific notations for $\cK^\fr$:
\begin{enumerate}
\item (Degeneration case) $\cK_\Gamma^\fr(W/B)$, 
\item (Singular case) $\cK_\Gamma^\fr(W_0)$,
\item (Pair case) $\cK_\Gamma^\fr(X,D)$.
\end{enumerate}
\end{definition} 
\begin{lemma}  The stack $\cK^\fr$ is open in $K^{tr}$. 
\end{lemma}
\begin{proof} It is enough to prove that it is stable under generalization. Assume we have a family of twisted transversal stable maps over $\Spec R$, with $R$ a discrete valuation ring, and that the fiber over the special point is $\fr$-twisted. Let $D_\ell$ be a splitting divisor on the general fiber; then it induces a splitting divisor on the special fiber, and since twisting indices are locally constant along deformations so long as the node (or the splitting divisor) doesn't get smoothed out, all the contact orders and the twisting of the splitting divisor are the same at the special point and at the general point of $\Spec R$.
\end{proof}

\begin{theorem}\label{twistprop} Let $\fr$ be a twisting choice. Then the canonical morphism  $\cK^\fr\to K^u_{pd}$ is proper.
\end{theorem} 
\begin{proof} First we claim that the morphism is of finite type locally over $K^u_{pd}$.
 Choose an open covering of  $K^u_{pd}$ where in each chart only finitely many contact orders  apear. Then objects in $\cK^\fr$ involve maps to targets with bounded twisting, and the stack of these targets is of finite type over the base. Since $\cK^\fr$ is proper over that stack the claim follows.

 \ChDan{It now suffices to use the valuative criterion of properness. 
 We take $\Delta=\Spec R$ for $R$ a discrete valuation ring, with generic point $\eta$ and closed point $s$. 
We} assume we are given a commutative diagram $$\xymatrix{
\eta\ar[r]\ar[d] &K^{\fr}\ar[d]\ar[r] & T^\tw\ar[d]\\
\Delta\ar[r] & K^u_{pd}\ar[r] &T^u
}$$
and we want to find a unique lifting $\Delta\to K^{\fr}$. 
We denote the induced families of stable maps as $(C,\Sigma,f:C\to \cW)$  and $(\tilde C_\eta,\tilde\Sigma_\eta, \tilde f_\eta:C_\eta\to \cW_\eta)$. 
The stack of twisted stable maps into $\cW^\tw_{univ}$ is proper over $T^\tw$. Therefore it is enough to show that there is a unique lifting $\Delta\to T^\tw$ such that the induced  family's central fiber $(\tilde C_s,\tilde\Sigma_s,\tilde f_s)$ is in $K^\fr$. 

Consider the set $\bDelta$ of multisets of positive integers. We have a canonical lifting of $K^u_{pd}\to T^u $ to a morphism $K^u_{pd}\to T^\bDelta$ to the stack of $\bDelta$-weighted expansions, see \ref{weighted-expansions},
where we weigh the expansions by the contact order at each splitting divisor. Similarly we have a canonical lifting of $K^\fr\to T^u $ to a morphism $K^\fr \to T^\fr$ to the stack of $\fr$-twisted $\bDelta$-weighted expansions, since the expansions in $K^\fr$ are by definition $\fr$-twisted $\bDelta$-weighted expansions. We obtain he following refinement of the previous commutative diagram:  
$$\xymatrix{
\eta\ar[r]\ar[d] &K^{\fr}\ar[d]\ar[r] & T^\fr\ar[d]\\
\Delta\ar[r] & K^u_{pd}\ar[r] &T^\bDelta.
}$$
Note that $T^\bDelta = T^\bone$ is the stack of $\bone$-twisted (namely untwisted) $\bDelta$-weighted expansions, so by Proposition \ref{P:untwisting-proper}
the morphism $T^\fr \to T^\bDelta$ is proper. Therefore, after a suitable base change we obtain a unique lifting $\Delta\to T^\fr$.
 By Lemma \ref{pd-is-pd} (3), any lift of $\{s\} \to K^u_{pd}\times_{T^u}T^\fr$ is in $K^\fr$, in particular the limit $(\tilde C_s,\tilde\Sigma_s,\tilde f_s)$ is in $K^\fr$, as needed.
\end{proof}
\begin{corollary} If $W/B$ is proper with projective coarse moduli space, then the stack $\cK^\fr$ is proper for every twisting choice $\fr$.
\end{corollary}

\begin{proof}
This follows from Corollary \ref{C:pd-proper}.
\end{proof}

 \section{Relative and degenerate Gromov-Witten invariants}\label{S:Invariants}
\subsection{Curve classes and orbifold cohomology on nodal DM stacks}

\ChDan{Gromov--Witten invariants require two type of homological entries: {\em curve classes $\beta$}, which affect the moduli space one constructs, and {\em cohomology insertions $\gamma_i$}, whose pullbacks are integrated over the moduli space and determine the invariants. Curve classes were introduced in Section \ref{S:Stablemaps-conventions}. Since we are considering invariants of stacks, the cohomological insertions we must use are elements $\gamma_1,\ldots,\gamma_n\in H^*_{orb}(W_0,\QQ) := H^*(\ocI(W_0),\QQ)$ of Chen and Ruan's   {\em orbifold cohomology}. Since all evaluation maps land in the smooth locus of an expansion of $W_0$, we need only consider classes in $H^*(\ocI(W_0),\QQ)$ supported on sectors meeting the smooth locus of $W_0$. } 

\subsection{Gromov--Witten invariants for nodal DM stacks}\label{DGWI}
In this Section \ref{DGWI} we fix a nodal, first order smoothable proper DM stack $W_0= X_1\sqcup_D X_1$ with a projective coarse moduli scheme.
We also fix data $\Gamma=(\beta,g,N,\bfe)$ as in \ref{C:deg-data}.

 We denote by $\cK^\fr_\Gamma(W_0)$ stack of  transversal, $\fr$-twisted stable maps to expansions of $W_0$. This stack  is proper by Theorem \ref{twistprop}. The structure morphism to  $\fT^\tw_0$  has a relative perfect obstruction theory, by a standard construction described in Section \ref{SS:obs-th}. Since $\fT^\tw_0$ has pure dimension zero, it has a natural fundamental class, and we can therefore define an induced virtual fundamental class $[\cK^\fr_\Gamma(W_0)]^\vir$. 

\begin{definition} \label{D:ev-deg}
The stack $\cK^\fr_\Gamma(W_0)$ carries a universal family of twisted stable maps to expansions of $W_0$ denoted as follows:
$$
\xymatrix{
\Sigma_i \ar@{^(->}[r]\ar[dr] &\cC^\tw\ar[r]^{f^\tw}\ar[d]^{p^\tw}& \cW^{\tw}_{0,univ}\\
& \cK^\fr_\Gamma(W_0). }
$$
We denote the underlying family of predeformable maps as follows:
$$
\xymatrix{
\Sigma_i \ar@{^(->}[r]\ar[dr] &\cC^u\ar[r]^{f^u}\ar[d]^{p^u}& \cW^{u}_{0,univ}\\
& \cK^\fr_\Gamma(W_0).}
$$

Composing with the morphism $\cW^{u}_{0,univ} \to W_0$ and stabilizing, we obtain
a diagram $$
\xymatrix{
\Sigma_i \ar@{^(->}[r]\ar[dr] &C\ar[r]\ar[d]& W_0\\
& \cK^\fr_\Gamma(W_0). }
$$
The markings $\Sigma_i$ remain the same since the curves $\cC^\tw,\cC^u$ and $C$ are isomorphic in a neighborhood of the $\Sigma_i$.

For $i\in N$ the morphism $\Sigma_i\to W_0$  induces an  evaluation morphism $\ev_i:\cK^\fr_\Gamma(W_0)\to \bcI_{e_i}(W_0)\subset \bcI(W_0)$, where $\bcI(W_0)$ is the rigidified inertia stack of $W_0$. See \cite[4.4]{AGV}.

Finally, the family of coarse curves $\underline C \to \cK^\fr_\Gamma(W_0)$ has sections $s_i:  \cK^\fr_\Gamma(W_0) \to \underline C$ induced by $\Sigma_i$. Following \cite[8.3]{AGV} we denote  
$$\psi_i \ =\  c_1\left(s_i^*\,\omega_{\underline{C}/\cK^\fr_\Gamma(W_0)}\right).$$
\end{definition}

\begin{definition} Assume for $i\in N$ we are given positive integers $m_i$ and cohomology classes
 $\gamma_i\in H^*(\ocI(W_0))$.
We define the Gromov--Witten invariant 
$$\left\langle \prod_{i\in N}\tau_{m_i}(\gamma_i) \right\rangle^{W_0}_{\Gamma}\  :=\ 
\deg  \left( \prod_{i\in N}(\psi_i ^{m_i} \cdot\ev_i^{*}\gamma_i)\cap [\cK^\fr_\Gamma(W_0)]^\vir\right).
$$
\end{definition}

\begin{remark} Recall that in $\fr$-twisted stable maps, the image of a marked point is never in a splitting divisor.  This implies that the evaluation maps only land in sectors on $\ocI(W_0)$ transversal to $D$. In particular Gromov--Witten invariants involving a class $\gamma_i$ from a sector supported in $\ocI(D)$ vanish.
\end{remark}

\subsection{Relative Gromov--Witten invariants}
\label{RGWI}
Here we fix a proper DM smooth pair $(X,D)$  with a projective coarse moduli scheme.
We also fix data $\Gamma=(\beta,g,N,M,\bfe,\bff,\bfc)$ as in \ref{C:pair-data}.

 We denote by $\cK^\fr_\Gamma(X,D)$ stack of  transversal, $\fr$-twisted  stable maps to expansions of $(X,D)$. Again this stack  is proper by Theorem \ref{twistprop}, and the structure morphism to  $\cT^\tw$  has a relative perfect obstruction theory described in Section \ref{SS:obs-th}, with induced virtual fundamental class $[\cK^\fr_\Gamma(X,D)]^\vir$. 

Evaluation maps $\ev_i: \cK^\fr_\Gamma(X,D) \to \cI(X)$ and classes $\psi_i$ for $i\in N$ are defined as in \ref{D:ev-deg}. Note however that for $j\in M$ the markings $\Sigma_j^\fr\subset \cC^\fr$ and $\Sigma_j\subset \cC^u$ are not isomorphic: $\Sigma_j^\fr \to \Sigma_j^u$ is a gerbe banded by $\bmu_{r/c_j}$, where $r=\fr(\{c_j\}_{ j\in M})$ is the twisting of the divisor.
An additional subtlety is the fact that the stack $\cD\to \cT^\tw$ is not a product, and we are only interested in the relative part of inertia.  Since the root markings $\Sigma_j$ map to $D$ and $\Sigma^\fr_j$ map to $\cD$ we can define the following:

\begin{definition}\label{D:ev-pair}
We denote by $\ev_j: \cK^\fr_\Gamma(X,D) \to \ocI_{f_j}(D)\subset \ocI(D)$ the evaluation map induced by $ \Sigma_j\to D$.
We denote by $\ev^\fr_j: \cK^\fr_\Gamma(X,D) \to  \ocI(\cD/\cT^\tw)$ the evaluation map induced by $\Sigma^\fr_j\to \cD$. Note that these land in $\ocI_{f_jr/c_j}(\cD/\cT^\tw)\subset  \ocI(\cD/\cT^\tw)$.
\end{definition}
The maps $\ev^\fr_j$ will play a role in the proof of the degeneration formula. Gromov--Witten invarians involve only $\ev_j$:

\begin{defn} Let $m_i,i\in N$ be nonnegative integers; $\gamma_i\in H^*(X),i\in N$ and $\gamma_j\in H^*(\ocI(D)), j\in M$. We define relative Gromov--Witten invariants with gravitational descendants by the formula
\begin{eqnarray*}
\lefteqn{
\left\langle \prod_{i\in N}\tau_{m_i}(\gamma_i)\big\vert\prod_{j\in M}\gamma_{j}\right\rangle_{\Gamma}^{(X,D)}}\\
&:=& \deg\left(\left(  \prod_{i\in N} \psi_i^{m_i}\cdot \ev_i^*\gamma_i\right)\cdot\left(\prod_{j\in M}{\ev_j}^{*}\gamma_{j}\right)\cap [\cK^\fr_\Gamma(X,D)]^\vir\right).
\end{eqnarray*}
\end{defn}
\begin{remark}
Note that unless the condition  \begin{equation}\label{goodass} \sum_{i\in M} d_i=\beta\cdot D\end{equation} is satisfied, the moduli stack $\cK^\fr_\Gamma(X,D)$ is empty and hence the invariant is zero. 
\end{remark}


\subsection{Independence of twisting choice}
 \begin{theorem}\label{gw-ind-twisting}
The Gromov--Witten invariants defined above are independent of the twisting choice $\fr$. \end{theorem}

The proof requires some preparation. Let $\bDelta$ be the set of multi-sets $\bc=\{c_1,\ldots,c_k\} $ of positive integers, and recall the stack $\fT_0^\fr\subset \fT_0^\bDelta$ of $\fr$-twisted, $\bDelta$-weighted  \ChDan{expansions}, defined in section \ref{weighted-expansions}.
 There is a natural morphism $\cK_\Gamma^\fr(W_0)\to \fT_0^\bDelta$ mapping each stable map to the labeling of each divisor by  the multiset of contact orders of the associated predeformable map. Its image is clearly contained in the open substack $\fT_0^\fr$. Similarly we have $\cK_\Gamma^\fr(X,D)\to \cT^\fr$.

\begin{proposition}\label{twist-compat-obs} Assume that $\fr$ and $\fr'$ are twisting choices with $\fr\prec\fr'$. Write for brevity $\cK\to T$  for $\cK_\Gamma^\fr(W_0)\to \fT_0^\fr$ (respectively, $\cK_\Gamma^\fr((X,D))\to \cT^\fr$)  and $\cK'\to T'$  for $\cK_\Gamma^{\fr'}(W_0)\to \fT_0^{\fr'}$, (respectively, $\cK_\Gamma^{\fr'}((X,D))\to \cT^{\fr'}$); similarly for the maps $\ev_i,\ev_i'$ and classes $\psi_i$, $\psi_i'$.
\begin{enumerate}
\item There is a natural morphism $\cK'\to\cK$.
\item This morphism induces a $2$--cartesian diagram 
$$\xymatrix{\cK'\ar[r]^\phi\ar[d] &\cK\ar[d]\\
T^{\fr'} \ar[r] & T^\fr,
}$$ 
where the lower arrow is given by Lemma \ref{compare-twisted}.
\item There is a natural isomorphism $\phi^*\EE_{\cK/T^\fr}\to \EE_{\cK'/T^{\fr'}}$.
\item We have $\psi_i'= \phi^*\psi_i$, and there is a natural equivalence between $\ev_i'$ and $\ev_i\circ\phi$ for $i\in N\sqcup M$.
\end{enumerate} 
\end{proposition}
\begin{proof} \begin{enumerate}
\item  Let $(C',\Sigma')\to \cW'\to W$ be a family of $\fr'$-twisted maps over a base scheme $S$. By Lemma \ref{compare-twisted}
there are morphisms  $T^{\fr'}\to T^\fr$ and correspondingly $\cW'\to \cW$. Consider the stabilization $(C,\Sigma)\to \cW$ of  the composition $(C',\Sigma')\to \cW'\to\cW$ (in the sense of \cite[Proposition 9.1.1]{Abramovich-Vistoli}). It is easy to see that this defines a family of $\fr$-twisted stable maps over $S$, and that this construction commutes with base change; hence, it defines a morphism $\cK'\to\cK$.
\item It is easy to check that the diagram is commutative. To construct a morphism from the fiber product to $\cK'$, assume that we are given a family of labeled  \ChDan{expansions}  $\cW'\to S$ (corresponding to a morphism $S\to T^{\fr'}$) and a family of twisted degenerate stable maps $(C,\Sigma)\to \cW$ over $S$, in such a way that $\cW\to S$ is induced by $\cW'\to S$ as a family of labeled twisted  \ChDan{expansions}. We then define $C':=\cW'\times_{\cW}C$, and $\Sigma'$ the inverse image of $\Sigma$; again, it is easy to check the required properties.
\item The diagram 
$$
\xymatrix{
C' \ar[r]\ar[d] & \cW'\ar[d] \\
C \ar[r] & \cW
}
$$ is Cartesian. The isomorphism of obstruction theories now follows directly from their definition.
\item The untwisted curve $(\cC^{u'},\Sigma')$ is the pullback of $(\cC^u,\Sigma)$. Both claims easily follow.
\end{enumerate}
\end{proof}

\begin{proof}[Proof of Theorem \ref{gw-ind-twisting}]
By Remark \ref{comptc} we may assume that $\fr\prec\fr'$. Using the notation in Proposition \ref{twist-compat-obs} we have that $\phi^*\EE_{\cK/\fT_0^\fr}\to \EE_{\cK'/\fT_0^{\fr'}}$. The morphism $\fT_0^{\fr'} \to \fT_0^{\fr}$ has degree 1 in the sense of \cite[Section 5]{Costello} by Proposition \ref{P:untwisting-proper}. By  \cite[Theorem 5.0.1]{Costello} we have an equality of the associated virtual fundamental classes $[\cK_\Gamma^\fr(W)]^\vir = [\cK_\Gamma^\fr(W)]^\vir$. The equality of invariants follows by the projection formula.
\end{proof}

\subsection{Invariance under twisting}
The following is used in \cite{ACW}:
\begin{proposition}
Let $\cX = X(\sqrt[r]{D})$ and $\cD = \sqrt[r]{N/D}$, and let $\pi: \cX \to X$ be the natural map. Then 
$$\left\langle \prod_{i\in N}\tau_{m_i}(\gamma_i)\big\vert\prod_{j\in M}\gamma_{j}\right\rangle_{\Gamma}^{(X,D)} = \left\langle \prod_{i\in N}\tau_{b_i}(\pi^*\gamma_i)\big\vert\prod_{j\in M}\pi^*\gamma_{j}\right\rangle_{\Gamma}^{(\cX,\cD)}
$$
\end{proposition}
\begin{proof} For fixed $\Gamma$ Let $\fr$ be a twisting choice giving a constant $r_0$ for all contact orders appearing in $\cK_{pd}$, and $\fr' = r\cdot \fr$. Then $\cK^{\fr'}_\Gamma(\cX,\cD)$ and  $\cK^{\fr}_\Gamma(X,D)$ are identical.
The result follows by the projection formula.

\end{proof}
\begin{remark}
A similar result holds for invariants of a singular variety.
\end{remark}

\subsection{Deformation invariance}
\label{definv}

Assume we are given a flat proper family $\pi:W\to B$ such that $B$ is a smooth curve, $b_0\in B$ is a point, $\pi$ is smooth over $B\setminus b_0$ and $W_0:=\pi^{-1}(b_0)$ is as in the previous section

\begin{theorem}\label{T:definv}
 The Gromov--Witten invariants of $W_b$ \ChDan{are independent of $b$, in the following sense: fix data $\Gamma = (\beta,\ldots)$ on $W$ and enumerate $\Gamma_t = (\beta_t,\ldots), t=1,\ldots $ where $\beta_t$ are all the distinct curve classes on $W_b$ whose image in $W$ is $\beta$.  Let $\gamma_i$ be cohomology classes on $W$, and denote their pullback to $W_b$ also by $\gamma_i$. Then the sum $$\sum_t \left\langle \prod_{i\in N}\tau_{m_i}(\gamma_i) \right\rangle^{W_b}_{\Gamma_t}$$ is independent of $b$.}
\end{theorem}
\begin{proof}
This follows from \cite[Proposition 7.2 (2)]{Behrend-Fantechi}, since the obstruction theory for $\cK^\fr_\Gamma(W/B)$ relative to $\fT_{B,b_0}$ restricts to that of $\cK^\fr_\Gamma(W_b)$ relative to $\fT_{0}$.
\end{proof}

\ChDan{
\begin{remark}
It would be  interesting to prove a version of this result which accounts for monodromy, and separates different $\beta_i$ and at the same time allows for cohomological insertions $\gamma_i$ supported on the fibers. 
\end{remark}
}

\subsection{Disconnected maps and invariants}\label{SS:disconnected-maps}
 We fix an abelian semigroup $H$. In the application we have in mind, $H$ will be the semigroup of effective curve classes.

\begin{defn}\label{D:modular-graph} An {\em  modular graph} $\Csi$ is a collection of vertices $V(\Csi)$, edges $E(\Csi)$  legs $L(\Csi)$ and roots $R(\Csi)$ with the usual relations and properties - here we divided the usual set of legs into two disjoint sets of legs and roots. These are weighted by the following data 
\begin{enumerate}
\item each vertex $v\in V(\Csi)$ is assigned an integer $g(v)\geq 0$ - the genus -  and an element $\beta(v)\in H$ -its weight.
\item each leg $l\in L(\Csi)$  is assigned an integer $e(l)>0$ - its index.
\item each root $r\in R(\Csi)$ is assigned two integers $f(l),c(l)>0$ - its index and contact order.
\end{enumerate}
The {\em total weight} $\beta(\Csi)$ of the graph is the sum of the weight of the vertices.   The {\em total genus} $g(\Csi)$ is given by the standard formula
$$2g(\Csi)-2 = \sum_{v\in V(\Csi)} (2g(v) -2) + 2\#E(\Csi).$$
Note that we do not assume the graph to be connected or even nonempty. The letter $\Csi$ is supposed to remind you of that. 
 
A labeling of the legs and roots by disjoint sets $M,N$ is the choice of a bijection  $M \leftrightarrow L(\Csi)$ and a bijection  $N\leftrightarrow R(\Csi)$.
\end{defn}   

\subsubsection{Moduli stacks and their properties}  
The theory developed in the last two sections applies to disconnected curves as well. Consider a possibly disconected twisted curve $C = \sqcup_{\nu=1}^h C_\nu$ with $C_\nu$ its  connected components of arithmetic genera $g_j$, each with marked points with labeling in disjoint  subsets $N_\nu$ and $M_\nu$ forming partitions of given ordered sets $N,M$. We  assume for simplicity that for each $\nu$ the set $M_\nu\cup N_\nu$ is nonempty. There is no loss of generality in assuming $N = \{1,\ldots,n\}$ and $M = \{n+1,\ldots,n+|M|\}$.   We assign each $C_\nu$ with a target curve class $\beta_\nu$ and package the data in the notation of a modular graph $\Csi = \sqcup \Gamma_\nu$ consisting of $h$ vertices assigned genera $g_\nu$ and weights  $\beta_\nu$, no edges, legs indexed by $N_\nu\subset N$ with weights $e_i,i\in N_\nu$ corresponding to the indices of these markings, and roots indexed by $M_\nu\subset M$ with similar weights $f_j, j\in M_\nu$.  We further  assume given  contact orders $\{c_j\vert j\in N\}$.

An {\em expanded $\fr$ twisted stable map} of type $\Csi$ into $(X,D)$ is a morphism $C \to \cX$ of type $\Xi$ into an  \ChDan{expanded pair}, with finite automorphism group over $X$, which is transversal and $\fr$ twisted. We again denote by $\cK_\Xi^\fr(X,D)$ the moduli stack of expanded $\fr$ twisted stable map of type $\Xi$. These stacks are algebraic and proper by the same arguments as in the connected case.
 \subsubsection{Disconnected Gromov--Witten invariants}
 The stack $\cK_\Xi^\fr(X,D)$ admits a perfect obstruction theory relative to $\cT^\tw$ as before. This gives rise to a virtual fundamental class.  We can use \ref{RGWI} to  construct Gromov--Witten invariants 
$ \left\langle \prod_{i\in N}\tau_{b_i}(\gamma_i)\big\vert\prod_{j\in M}\gamma_{j}\right\rangle_{\Xi}^{(X,D)}$ with the exact same formula.

\subsubsection{Contraction morphisms}
Given an $\fr$-twisted stable map $C \to \cX$ of type $\Xi$, we have for each $\nu$ a map $C_\nu\to \cX$. This map is not necessarily stable, since some components of $\cX$ require other $C_{\nu'}$ to stabilize them; it is not necessarily $\fr$-twisted since some contact orders are removed. There is a canonical way to stabilize  $C_\nu\to \cX$ as follows: 
\begin{enumerate}
\item Let $\cX \to \cX^\fr$ the canonical partial untwisting associated to the contact orders appearing in $C_\nu$ and the twisting choice $\fr$. Let $C_\nu^\fr\to \cX^\fr$ be the associated twisted stable map as in \cite[Corollary 9.1.2]{Abramovich-Vistoli}.
\item By definition, any semistable component $\PP$ of $\cX^\fr$ has the same twisting on its two boundary divisors. It follows from \cite[Section 9.2]{Abramovich-Vistoli} that there is a canonical contraction $\cX^\fr\to\bar\cX^\fr$ of all semistable components. We again define $\bar C_\nu^\fr\to \bar\cX^\fr$ as in \cite[Corollary 9.1.2]{Abramovich-Vistoli}.
\end{enumerate}

This defines a {\em contraction morphism} $\cK^\fr_\Xi(X,D) \to \cK_{\Gamma_\nu}^\fr(X,D)$. Combining these we obtain  a morphism $$\eps:\cK_\Xi^\fr(X,D) \to\prod_{\nu=1}^h \cK_{\Gamma_\nu}^\fr(X,D).$$

\begin{proposition}\label{P:disc-compare}
 $$
\eps_*[\cK_\Xi^\fr(X,D)]^{vir}\ \ =\ \ \prod_{\nu=1}^h[\cK_{\Gamma_\nu}^\fr(X,D)]^{vir}.$$
\end{proposition}

This immediately implies a result on Gromov--Witten invariants.
Assume all classes $\gamma_i$ have homogeneous parity and consider the sign $(-1)^\epsilon$ determined formally by the equality
$$\prod_{i\in N}\gamma_i\,\cdot\,\prod_{j\in M}\gamma_j\ \  =\ \  (-1)^\epsilon \prod_{\nu=1}^h
\left(\prod_{i\in N_\nu}\gamma_i\,\cdot\,\prod_{j\in M_\nu}\gamma_j\right)$$
\begin{corollary}
$$ \left\langle \prod_{i\in N}\tau_{b_i}(\gamma_i)\big\vert\prod_{j\in M}\gamma_{j}\right\rangle_{\Xi}^{(X,D)}\ \  = \ \ (-1)^\epsilon \prod_{\nu=1}^h 
 \left\langle \prod_{i\in N_\nu}\tau_{b_i}(\gamma_i)\big\vert\prod_{j\in M_\nu}\gamma_{j}\right\rangle_{\Gamma_\nu}^{(X,D)}$$
\end{corollary}

\begin{proof}
This follows from the projection formula.
\end{proof}

\subsubsection{An auxiliary stack of expansions}
We prove Proposition \ref{P:disc-compare} by applying the main technical result, Theorem 5.0.1 of \cite{Costello}. This needs some preparation. 

Let $\cT'$ be the stack defined as follows. $\cT'(S)=\{\cD\times S\subset \cX \stackrel{\eps_\nu}{\to} \cX_\nu \stackrel{\theta_\nu}{\to} X\times S\}$ where:
\begin{enumerate}
 \item  $\cD\times S\to \cX_i\to X\times S$ is an \ChDan{expanded pair} for each $\nu=1,\ldots,h$;
\item there is a morphism $\rho:\cX\to X\times S$ such that $\rho$ is isomorphic to $\theta_\nu\circ \eps_\nu$ for each $\nu=1\ldots,h$;
\item $\cD\times S\to \cX\stackrel{\rho}{\to} X\times S$ is an \ChDan{expanded pair};
\item for each $\nu=1,\ldots,h$ the morphism $\eps_\nu$ is a partial contraction;
\item for each $\PP_i$ in $\cX$ there exists at least one $1\le \nu\le h$ such that $\eps_\nu|_{\PP_i}$ is an isomorphism with its image. 
\end{enumerate}

There is a natural forgetful morphism $\cT'\to (\cT^\tw)^h$ given by $$
\{\cD\times S\subset \cX \to \cX_\nu \to X\times S\}\mapsto \{\cD\times S\subset \cX_\nu {\to} X\times S\} _{\nu=1,\ldots,h}.
$$
This map has degree 1 in the sense of \cite[Section 5]{Costello}, since both source and target have $\Spec \CC$ as a dense open substack.

\begin{lemma}
 There is a natural cartesian diagram $$
\xymatrix{
\cK_{\Xi}\ar[r]\ar[d] & \prod \cK_{\Gamma_\nu}\ar[d]\\
\cT' \ar[r] & \cT^h
}
$$
\end{lemma}
\begin{proof} That the diagram is commutative is obvious. To prove that it is cartesian,
let $S$ be a base scheme. Let $\{(C_\nu,f_\nu,\cD\times S\to \cX_\nu\to X\times S\}_{\nu=1\ldots,h}$ be an object of $\prod \cK_{\Gamma_\nu}(S)$, and $\{\cD\times S\subset \cX \to \cX_\nu \to X\times S\}$ an object in $\cT'(S)$. Define curves $C_\nu'$ by $C_\nu':=C_\nu\times_{\cX_\nu}\cX$; lift the marked gerbes of $C_\nu'$ to $\tilde C_\nu'$ in the unique possible way if they do not map to $\cD\times S$ and respecting the map to $\cD\times S$ otherwise. Let $C:=\sqcup C_\nu'$ and  $f:C\to \cX$ be the morphism such that $f|_{C_\nu'}$ is induced by the fiber product. We leave it to the reader to check that this provides an inverse to the natural map from $\cK_{\Xi}$ to  the fiber product of $\cT'$ and $\prod \cK_{\Gamma_\nu}$ over $\cT^h$.  
\end{proof}

Let $\cT'\to \cT$ be the forgetful morphism defined by $$
\{\cD\times S\subset \cX \to \cX_\nu \to X\times S\}\mapsto \{\cD\times S\subset \cX {\to} X\times S\}.$$

\begin{lemma} The morphism $\cT'\to \cT$ so defined is \'etale.
\end{lemma}
\begin{proof} This is identical to \cite{ACW}\Barbara{find reference}.
\end{proof}

The construction in Section \ref{obs_th} give relative obstruction theories for the morphisms $\cK_{\Gamma_\nu}\to \cT$, hence for the morphism $\prod \cK_{\Gamma_\nu}\to \cT^h$, and for $\cK_\Xi\to \cT'$.
\begin{lemma} 
The obstruction theory for $\cK_\Xi/\cT'$ is the pullback of the obstruction theory $\prod \cK_{\Gamma_\nu}\to \cT^h$.
\end{lemma}
\begin{proof}
Write $\cK'$ for $\cK_\Xi$, and fix an index $\nu$. We consider the commutative diagram $$
\xymatrix{
\cC'_\nu\ar[d]_p\ar[r]^{g} &\cX\ar[d]^q\\
\cC_\nu\ar[d]\ar[r]^{f} & \cX_\nu\\
\cK'&
}
$$
where $\cC_\nu$ is the pullback to $\cK'$ of the universal curve over $\cK_{\Gamma_\nu}$ and $\cC_\nu'$ is the corresponding component of the universal curve over $\cK'$, together with the structure maps. Note that $\cC'_\nu\to \cC_\nu$ is a partial stabilization map, i.e. it is locally a base change of some forgetful morphism $\bar M_{g,n+k}^\tw\to \bar M_{g,n}^\tw$ . This implies that $\RR\pi_*\cO_{\cC'}=\cO_C$, and therefore that $\RR\pi_*\circ L\pi^*:D(\cC)\to D(\cC)$ is the identity morphism. 
On the other hand, the fact that the square in the diagram is cartesian shows that the pullback of the complex $\LL$ from $\cC_\nu$ is the corresponding complex for $\cC_\nu'$.
\end{proof}
\begin{proof}[Proof of Proposition \ref{P:disc-compare}]
This is now immediate from \cite[Theorem 5.0.1]{Costello}, as we have shown that the obstruction theories are compatible and the map $\cT'\to (\cT^\tw)^h$ is of pure degree 1. 
\end{proof}

 \section{Degeneration formula}\label{S:Degeneration}
\subsection{Setup}
We fix a variety $W_0=X_1\sqcup_D X_2$ with first-order smoothable singularity along $D$ dividing it in two smooth pairs $(X_1,D) $ and $(X_2,D)$. We let $H$ be the monoid of curve classes on $W_0$, and $H_1, H_2$ the coresponding monoids on $X_1,X_2$. We view $H_1,H_2$ as submonoids of $H$.

We Notice that, although $H^*_{orb}(W_0) = H^*(\ocI(W_0))$ has a rational degree shifting, when we consider parity we always refer to the unshifted grading.

In this section we will keep fixed the notation introduced in Subsection \ref{DGWI}. In particular we fix  $\Gamma = (g,N,\beta,\bfe)$ as in Convention \ref{C:deg-data}, which we may view as a modular graph with one vertex with genus $g$ and weight $\beta$, a curve class on $W_0$  and legs labelled by an ordered set $N$ and marked with indices $e_i$.   We will also fix cohomology classes $$\gamma_i\in H^*(\ocI(W_0)), i\in N$$ with homogeneous parity, and nonnegative integers $m_1,\ldots,m_n$. These will be used for insertions
in Gromov--Witten invariants. In this section we may take $N= \{1,\ldots,n\}$.

\begin{defn}\label{D:enh-spl} A{\em splitting} $\eta$ of $\Gamma$  is an ordered  pair  $\eta=(\Csi_1,\Csi_2)$ where
\begin{enumerate}
\item $\Csi_1$ and $ \Csi_2$ are modular graphs as in Definition \ref{D:modular-graph} with no edges or loops,
\item The labelling of legs $L(\Csi_1)\leftrightarrow N_1$ and $L(\Csi_2)\leftrightarrow N_2$ form a partition $N_1\sqcup N_2 = N$.
\item The labelling of roots $R(\Csi_1)\leftrightarrow M$ and $R(\Csi_2)\leftrightarrow M$ are in the same ordered set disjoint from $N$, which can be safely taken as $\{n+1,\ldots,n+|M|\}$.
\item A leg $l\in L(\Csi_1)\cup L(\Csi_2)$ corresponding to $i\in N$  is assigned the corresponding index $e_i$.
\item A root $r\in R(\Csi_1)$  corresponding to $j\in M$  is assigned index $f_j$ and contact order $c_j$. A root   $r\in R(\Csi_2)$ is assigned {\em the same} corresponding index $f_j$ and contact order $c_j$. 
\item A vertex $v\in V(\Csi_1)$ is assigned genus $g(v)$ and weight $\beta(v)\in H_1$; similarly the weight of  $v\in V(\Csi_2)$ is a curve class  $\beta(v)\in H_2.$
\end{enumerate}

We define  $\mathbf{d}=(d_j,j\in M)$ by $d_j=c_j/f_j$. These are in general rational numbers, which we call {\em intersection multiplicities}. We denote the set of roots incident to a vertex $v$ by $R(v)$ and identify it with a subset of $M$.

These data must satisfy the following conditions:
\begin{enumerate}
\item[A.] The graph $\Gamma$ obtained by gluing $\Csi_1$ and $\Csi_2$ along the legs labeled by $\{n+1,\ldots,n+n_D\}$ is connected of genus $g$ and total weight\footnote{Of course here we are \ChDan{tacitly identifying an element of $H_2(X_e)$} with its image in $H_2(X)$.} $\beta$.
\item[B.] For every vertex $v\in V(\Csi_1)$, one has $$
\sum_{j\in R(v)}d_j=(\beta(v)\cdot D)_{X_1}.$$ Similarly if $v\in V(\Csi_2)$ then $$
\sum_{j\in R(v)}d_j=(\beta(v)\cdot D)_{X_2}.$$
\end{enumerate}
\end{defn} 
\begin{remark}
 Let $\beta_1$ be the total weight of $\Csi_1$ and $\beta_2$ the total weight of $\Csi_2$. Then (B) implies that $(\beta_1\cdot D)_{X_1} = (\beta_2\cdot D)_{X_2}$.
\end{remark}

\begin{remark}
The distinction between intersection multiplicities $d_j$ and contact orders $c_j$ is a feature of the orbifold situation, the ratios $f_i = c_i/d_i$ being the indices of the corresponding marked points mapping to $D$.  We see in \ref{S:pf-mainth} that  one can avoid the need for the $c_i$ in the forumla as stated in \ref{maintheorem}, but our proof requires using them.\end{remark}

\begin{defn} An isomorphism of splittings $(\Csi_1,\Csi_2)\to (\Csi_1',\Csi_2')$ is an isomorphism of modular graphs respecting the labellings, in particular the orders of $N$ and $M$. 

We denote by $\Omega(\Gamma)$  the set  of isomorphism classes of splittings of $\Gamma$. 
\end{defn}

\begin{remark} Passing to isomorphism classes is harmless:
since by assumption the  glued graph is connected, every vertex in $\Csi_1,\Csi_2$ is incident to at least one root, and since the roots are labelled by an ordered finite set, the automorphism group of a splitting is trivial. So  the groupoid of splittings is rigid and therefore equivalent to the set  $\Omega(\Gamma)$.
\end{remark}

\begin{defn}
The symmetric group $S(M)$ acts on $\Omega$ by its action on $M$. Two splittings are said to be {\em equivalent} if they belong to the same $S(M)$-orbit. 
We let $\overline\Omega$ be the set of equivalence classes.
\end{defn}
\begin{defn} Fix a twisting choice $\fr$. For each $\eta\in \Omega$, we define
$$
r(\eta):=\fr(\bc).$$
\end{defn}
\begin{definition}
Consider the {\em standard} pairing 
$$\begin{array}{ccl} 
H^*(\ocI(D)) \times H^*(\ocI(D)) &\to &  \QQ \\
(\theta_1,\theta_2) & \mapsto & \int_{\ocI(D))}\theta_1\cdot\theta_2.
\end{array}
$$
Let $F$ be a basis of $H^*(\ocI(D))$ of classes with homogeneous parity. For each $\delta\in F$ we denote by $\delta^\vee$ be the dual element in the dual basis with respect to this pairing. {\em In order to avoid issues of signs we define $\delta^\vee$ to be dual to $\delta$ if $\int_{\ocI(D))}\delta^\vee\cdot\delta = 1$ in this order} - this ensures that the Poincar\'e dual class of the diagonal of $\ocI(D) $ is $\sum_{\delta\in F} \delta\times \delta^\vee$.
\end{definition}

\subsection{Statement of the formula} \label{SS:Degeneration}

Here is the degeneration formula the way it naturally arises in our proof:
\begin{theorem}\label{T:degeneration}  For any choice of nonegative integers $m_1,\ldots,m_n$, and cohomology classes $\gamma_i\in H^*(I(W_0))$, the following degeneration formula holds:
\begin{eqnarray*}
\lefteqn{\left\langle\prod_{i=1}^n \tau_{m_i}(\gamma_i) \right\rangle_{\beta,g}^{W_{b_0}} }\\
&=
\sum_{\eta\in \Omega}\frac{\prod_{j\in M} c_j}{|M|!}\ \sum_{\delta_i\in F} (-1)^{\epsilon}&
\left\langle\prod_{i\in N_1}\tau_{m_{i}}(\gamma_{i})\big\vert\prod_{i\in M}\delta_i\right\rangle^{(X_1,D)}_{\Gamma_1} \\ &&\cdot
\left\langle\prod_{i\in N_2}\tau_{m_{i}}(\gamma_{i})\big\vert\prod_{i\in M}\iota^*\delta^\vee_i\right\rangle^{(X_2,D)}_{\Gamma_2}.
\end{eqnarray*}
The sign $ (-1)^{\epsilon}$ is fixed in terms of the parity of the classes so that formally the following holds:
$$\prod_{i\in N}\gamma_i\cdot\prod_{j\in M} \delta_j\delta_j^\vee = (-1)^{\epsilon}\prod_{i\in N_1} \gamma_{i} \prod_{i\in M}\delta_i \prod_{i\in N_2}\gamma_{i}\prod_{i\in M}\delta_i^\vee.$$
\end{theorem}

\begin{remark}
In \cite{Li2}, one sums over the  set of equivalence classes $\overline\Omega$ of splitting types, and therefore the factor $|M|!$ in the denominator is replaced by $|Eq(\eta)|$, the stabilizer of $\eta$ inside $S_{|M|}$, introduced in \cite[p. 574]{Li1}, \cite[p. 203]{Li2}.     
\end{remark}

\subsubsection{The Chen--Ruan pairing and Theorem \ref{maintheorem}}\label{S:pf-mainth}
	As in \cite{AGV}, Section 6.4, or \cite{Chen-Ruan}, the formalism becomes a bit more elegant if one uses the Chen--Ruan pairing. Here one treats the evaluation maps $\ev_j, j\in M$ as if their target is $\cI(D)$ rather than $\ocI(D)$, and further includes $\iota$ in the pairing. In our situation, if we identify $\delta\in F$ with its pullback in $H^*(\cI(D))$, the dual element with respect to the {\em standard} pairing of $H^*(\cI(D))$ becomes $r\cdot \delta^\vee$. Further, we can change this pairing by applying $\iota$ on the right element, namely use
$$\begin{array}{ccc} 
H^*(\cI(D)) \times H^*(\cI(D)) &\to &  \QQ \\
(\theta_1,\theta_2) & \mapsto & \int_{\cI(D))}\theta_1\iota^*\theta_2,
\end{array}
$$
equivalently
$$\begin{array}{ccc} 
H^*(\ocI(D)) \times H^*(\ocI(D)) &\to &  \QQ \\
(\theta_1,\theta_2) & \mapsto & \int_{\ocI(D))}\frac{1}{r}\theta_1\iota^*\theta_2,
\end{array}
$$
obtaining the Chen Ruan pairing.
Then the dual element of $\delta$ with respect to the Chen--Ruan pairing is $\tilde\delta^\vee = r\cdot\iota^* \delta^\vee$.  Note again that the duality is defined so that $\int_{\cI(D))} \iota^*\tilde\delta^\vee \delta = 1$ to avoid signs in the decomposition of the class of the diagonal.

Theorem \ref{maintheorem} follows as a version of the Theorem \ref{T:degeneration} above, in which the contact orders $c_i$ are not used but  the more invariant intersection multiplicities $d_i=c_j/f_j$ instead. Indeed the pullback under the evaluation map $\ev_j:\cK^\fr_{\Csi_i}(X_i,D)\to \ocI(D)$ of the involution-invariant locally constant factor $r$ is the index $f_j$. This gives 
$$\prod_{j\in M} c_j\ev_j^*\iota^*\delta_j^\vee = 
\prod_{j\in M} d_j\ev_j^*\iota^*\tilde\delta_j^\vee$$
as required.

\subsection{Outline of proof of Theorem \ref{T:degeneration}} Fix a twisting choice $r$ and write $\cK$ for $\cK^\fr_\Gamma(W,\pi)_{b_0}$. The proof goes in several steps. These will be completed in the next sections, as follows:

\noindent {\sc Sections \ref{Sec:split-coarse}, \ref{Sec:split-stack-target}, \ref{Sec:decompose-moduli}:} for $\eta\in\Omega$ we define a proper Deligne--Mumford stack $\cK_\eta$ parametrizing maps to a twisted  \ChDan{expansion} with a fixed splitting divisor of type $\eta$, together with a morphisms $\st_\eta:\cK_\eta \to \cK$. We prove (Proposition \ref{L:virt-split})   
$$
[\cK]^\vir=\sum_{\eta\in\Omega}\frac{r(\eta)}{|M|!}\st_{\eta*}[\cK_{\eta}]^\vir.$$

\noindent {\sc Section \ref{Sec:gluing-target}:}
 fix $\eta=(\Csi_1,\Csi_2) \in\Omega$, and let $\cK_{\Csi_1}$ and $\cK_{\Csi_2}$ be the moduli stacks of relative stable maps corresponding to $\Csi_1$ and $\Csi_2$ respectively. On $\cK_{\Csi_1} \times \cK_{\Csi_2}$ there is a canonical gerbe banded by $\bmu_{r(\eta)}$, which we denote $u_\eta:\cK_{1,2} \to \cK_{\Csi_1} \times \cK_{\Csi_2}$, which parametrizes pairs of twisted stable maps together with the data of a glued target. 

\noindent {\sc Sections \ref{Sec:gluing-source}-\ref{Sec:coarse-diagonal}:}
We construct a commutative diagram with cartesian square
$$
\xymatrix{\cK_\eta\ar[r]^{q_\eta}\ar[dr]&\cK_\eta^*\ar[r]\ar[d]&\cK_{1,2}\ar[d]\\
&\ocI(D)^{M}\ar[r]_(.35){\Delta}&(\ocI(D)\times \ocI(D))^M}$$
where  the morphism on the right is the product of the two evaluation maps in $M$, and $\Delta$ is the diagonal with the second factor composed with $\iota$; In Proposition \ref{P:diag-gysin} we prove that $$
q_{\eta*}[\cK_\eta]^\vir=\left(\prod_{j\in M} c_j\right)\cdot \Delta^![\cK_{1,2}]^\vir\in A_*(\cK^*_\eta).$$

\noindent {\sc Section \ref{SS:endpf}:} The degeneration formula follows by another application of the projection formula.

\subsection{Splitting the coarse target}\label{Sec:split-coarse} We form the following cartesian diagram:

$$
\xymatrix{
\cK_{\fQ}\ar^s[r]\ar[d] 	& \cK\ar[d]  \\
\fQ\ar[r]\ar[d] 		& \fT^\tw_0 \ar[d] \\
\fT_0^{u,\spl} \ar[r] 			&\fT^u_0  \\
}
$$


The stack $\fT_0^{u,\spl}$  of untwisted  \ChDan{expansions} with a choice of splitting divisor is defined in Section \ref{SS:spl-exp}. It is nonsingular, and coincides with the normalization of $\fT^u_0$. 

The stacks $\fQ$ and $\cK_{\fQ}$  are formed as the fibered products making the  diagram cartesian. 
Therefore the perfect obstruction theory of $\EE^\bullet_{\cK/\fT_0^\tw}$ pulls pack to a perfect obstruction theory $\EE^\bullet_{\cK_{\fQ}/\fQ}$ defining a virtual fundamental class which we denote $[\cK_{\fQ}]^\vir$.  

Since $\fT_0^{u,\spl}\to \fT^u_0$ is the normalization of a reduced normal crossings stack, it has pure degree 1 in the sense of \cite[Section 5]{Costello}. Since $\fT_0^\tw \to \fT_0^u$ is flat, it follows that the morphism $\fQ\to \fT_0^\tw$ is of pure degree 1 in the same sense as well.  We have the following:

\begin{lemma}\label{L:coarse-split}
$$s_* [\cK_{\fQ}]^\vir = [\cK_{b_0}]^\vir$$
\end{lemma}

\begin{proof}
This follows from \cite[Theorem 5.0.1]{Costello}, see also
 \cite[Proposition 2, Section 4.3]{Manolache}. 
\end{proof}

\subsection{Splitting the stack target}\label{Sec:split-stack-target}

In Section \ref{SS:spl-exp}
we introduced a natural decomposition of $\fQ = \fT_0^{u,spl}\times_{\fT_0^u} \fT_0^\tw$ into open and closed loci according to the twisting index of the twisted  \ChDan{expansion} along the chosen singular component:
$$\fQ = \coprod_{r\geq 1}\fQ_r$$
and accordingly we have a decomposition 
$$\cK_{\fQ} = \coprod_{r\geq 1}\cK_{\fQ_r}.$$

The stack $\fQ_r$ is nonreduced. The reduced substack is the smooth stack $\fT^{r,\spl}_{0}$, the stack of twisted  \ChDan{expansions} with splitting divisor of index $r$. 

By Lemma \ref{L:tw-1/r}
the morphism  $\fT^{r,\spl}_{0}\to \fQ_r$ is of degree $1/r$, in the sense that the image of  $[\fT^{r,\spl}_{0}]$ is $r^{-1} [\fQ_r]$.
This is sufficient for applying   \cite[Theorem 5.0.1]{Costello} in Manolache's version \cite[Proposition 2, Section 4.3]{Manolache}. We therefore obtain the following:
\begin{lemma}\label{Lem:split-target-r}
Consider the fiber diagram
$$\xymatrix{
\cK_r^\spl\ar[r]^{t_r} \ar[d] & \cK_{\fQ_r}\ar[d] \\ 
\fT^{r,\spl}_{0}\ar[r]&\fQ_r
}$$
Then $$[\cK_{\fQ_r}]^{vir} \ \ = \ \ r\ \cdot \  (t_r)_*[\cK_r^\spl]^{vir}.$$
\end{lemma}

The multiplicity $r$ in this lemma depends on the twisting choice, since the formation of the moduli spaces does. It is important to notice that at the end it will be cancelled by that appearing in Lemma \ref{Lem:glue-target} below.

\subsection{Decomposing the moduli space with split target}\label{Sec:decompose-moduli}


Recall that we denote by $\overline\Omega = \Omega/{\sim}$ the set of equivalence classes of splitting types under the action of the symmetric group $S(M)$, and by $\bar\eta$ the equivalence class of $\eta\in \Omega$. 

Given a positive integer $r$, denote by $\Omega_r$ the set of isomorphism classes of types $\eta$ satisfying $r(\eta) = r$, and by $\bar\Omega_r$ the set of equivalence classes.
We can now refine the decomposition as follows:
$$  \cK^\spl_r = \coprod_{\bar\eta\in \bar\Omega_r} \cK^\spl_{\bar\eta}.$$

Denote by $t_{\bar\eta}:\cK^\spl_{\bar\eta} \to \cK_{\fQ_r}$ the restricted morphism. On the level of virtual fundamental classes, Lemma \ref{Lem:split-target-r} gives

$$[\cK_{\fQ_r}] \ \ = \ \ r\ \cdot \  (t_{\bar\eta})_*\sum_{\bar\eta\in \bar\Omega_r}[\cK_{\bar\eta}^\spl].$$

Now denote by $\cK_{\eta}\to \cK_{\bar\eta}$ the cover obtained by labeling the distinguished nodes of the source curve by the set $M$.  This is clearly an  $S_{|M|}$-bundle, and therefore it has an associated perfect obstruction theory and virtual fundamental class. Denote by $t_\eta:\cK_{\eta}\to \cK_{\fQ_r}$ the composite map.

Putting Lemmas \ref{L:coarse-split} and \ref{Lem:split-target-r} together we obtain

\begin{proposition}\label{L:virt-split}
$$[\cK] \ \ = \ \ \sum_{\eta\in \Omega}\  \frac{r(\eta)}{|M|!}\ \cdot \  (s\circ t_\eta)_*[\cK_\eta^\spl].$$
\end{proposition}

\subsection{Gluing the target} \label{Sec:gluing-target}

Recall from \ChDan{Section \ref{Sec:pairs}} that for an integer $r$ (not to be confused by the implicit twisting choice) we denote by $\cT^\tw_r \subset \cT^\tw$ the substack of relative twisted expanded degenerations with twisting index $r$ along $D$.
In Lemma \ref{L:gl-1/r} we considered the natural morphism
$\fT_{0}^{r,\spl} \to \cT^\tw_r \times \cT^\tw_r$
corresponding to the two components of the partial normalization of the universal family $\cW^\tw_{0,univ}$ and showed that it  is a gerbe banded by $\bmu_r$.

We begin approaching $\cK_\eta^\spl$ from the other direction, namely from stacks of relative stable maps to the components of $W_0$. Given $\eta=(\Csi_1,\Csi_2)$ we denote $r=r(\eta)$, and use the shorthand notation $\cK_{\Csi_1} = \cK^\fr_{\Csi_1}(X_1,D)$ and $\cK_{\Csi_2} = \cK^\fr_{\Csi_2}(X_2,D)$.

\begin{definition}
We define $\cK_{1,2}$ by the following fiber diagram:
$$
\xymatrix{
\cK_{1,2}\ar^(0.4){u_\eta}[r]\ar[d]& \cK_{\Csi_1} \times \cK_{\Csi_2}\ar[d] \\
\fT_{0}^{r,\spl}\ar[r] & \cT^\tw_r \times \cT^\tw_r.
}$$
The stack $\cK_{1,2}$, which depends on $\eta$,  parametrizes a glued {\em twisted} target, along with a pair of relative stable maps of types  $\Csi_1$ and $\Csi_2$ to the two parts of the twisted target. 

Composing with the projections, we have morphisms $u_{\eta 1}:\cK_{1,2} \to \cK_{\Csi_1}$ and $u_{\eta 2}:\cK_{1,2} \to \cK_{\Csi_2}$
\end{definition}

Recall (Lemmas \ref{obsth}, \ref{itsperf}) that we have perfect obstruction theories $\EE_{\cK_{\Csi_1}/\cT^\tw_r}\to\LL_{\cK_{\Csi_1}/\cT^\tw_r}$ and $\EE_{\cK_{\Csi_2}/\cT^\tw_r}\to\LL_{\cK_{\Csi_2}/\cT^\tw_r}$. These are defined as follows: consider the universal relative twisted stable map
$$\xymatrix{
\cC_1 \ar[d]_{p_1}\ar[r]^{f_1} & \fX_1 \\
\cK_{\Csi_1}
}$$
Denote by $P_1\subset \cC$ the divisor given by the leg markings. 

Consider the complex $\LL_{\Box_1} :=[f_1^*\LL_{\fX_1/\cT_{r}^\tw} \to \Omega^1_{\cC_1/\cK_{\Csi_1}}(\log P_1)]$. We  have a perfect relative obstruction theory on $\cK_{\Csi_1} /\cT_{r}^\tw$ given by taking the complex
$$\EE^\bullet_{\cK_{\Csi_1} /\cT_{r}^\tw} = (\bR p_{1*}(\LL_{\Box_1}^\vee))^\vee[-1]$$
with its natural map to $\LL_{\cK_{\Csi_1}/\cT_{r}^\tw}$. The construction for $\EE^\bullet_{\cK_{\Csi_2} /\cT_{r}^\tw}$ is identical.

Combining these, we have a perfect obstruction theory $\EE^\bullet_{\cK_{\Csi_1}/\cT^\tw_r}\oplus \EE^\bullet_{\cK_{\Csi_2}/\cT^\tw_r}$ on $\cK_{\Csi_1} \times \cK_{\Csi_2} / \cT^\tw_r\times \cT^\tw_r$. As the morphism $\fT_{0}^{r,\spl} \to \cT^\tw_r \times \cT^\tw_r$ is \'etale, so is the morphism $\cK_{1,2} \to \cK_{\Csi_1} \times \cK_{\Csi_2}$, and the pullback of the same complex gives a perfect obstruction theory for $\cK_{1,2}/\fT_{0}^{r,\spl}$. We denote by $[\cK_{1,2}]^\vir$ and $[\cK_{\Csi_1} \times \cK_{\Csi_2}]^\vir=[\cK_{\Csi_1}]^\vir \times [\cK_{\Csi_2}]^\vir$ the associated virtual fundamental classes. Since the degree of $u_\eta:\cK_{1,2} \to \cK_{\Csi_1} \times \cK_{\Csi_2}$ is $r^{-1}$ we obtain the following:

\begin{lemma}\label{Lem:glue-target}
$$[\cK_{\Csi_1} \times \cK_{\Csi_2}]^\vir \ \ = \ \ r\ \cdot \ (u_\eta)_*  [\cK_{1,2}]^\vir.$$
\end{lemma}
Notice that the multiplicity $r=r(\eta)$ obtained here, which depends on the twisting choice, coincides with the multiplicity appearing in Lemma \ref{Lem:split-target-r}. In the comparison of invariants this multiplicity  cancels out.

Denote by $\cD$ the universal boundary divisor over $\cK_{1,2}$. It is a gerbe banded by $\bmu_r$ over the coarse boundary divisor $\cK_{1,2}\times D$.

\subsection{Gluing the source} \label{Sec:gluing-source} There is a natural morphism 
$$v_\eta:\cK_\eta \to \cK_{1,2}$$
obtained by associating to a map $\cC \to \cW^\tw_0$ with splitting of type $\eta$ the two maps $\cC_1 \to \cX_1 \hookrightarrow \cW^\tw_0$ of type $\Csi_1$ and $\cC_2 \to \cX_2 \hookrightarrow  \cW^\tw_0$ of type $\Csi_2$ with source curves determined by the splitting. We now put this in a fiber diagram and demonstrate the compatibility of the given perfect obstruction theories.

Recall from Definition \ref{D:ev-pair}  that the restriction of a stable map $f_1:\cC_1 \to \cX_1$ to   $\Sigma_{1j}$   gives rise to the evaluation map ${\ev^\fr_j}_{\Csi_1}: \cK_{\Csi_1} \to \ocI(\cD/\cT^\tw)$.


Composing with $u_{\eta 1}\circ v_\eta: \cK_\eta \to  \cK_{\Csi_1}$ 
 denote  
the product  morphisms 
$${\bev}^\fr_\eta = \prod_{j\in M} {{\ev^\fr_{j}}_{\Csi_1}}\circ u_{\eta 1}\circ v_\eta: \cK_{\eta} \to \ocI(\cD/\cT^\tw)_{\cT^\tw}^M.$$ This notation means that we are taking the $M$-th fibered product over $\cT^\tw$. Since this notation is cumbersome we use the shorthand $$\ocI^M :=\ocI(\cD/\cT^\tw)_{\cT^\tw}^M.$$
 Also denote 
 $${\bev}^\fr_{1,2} = \prod_{j\in M} ({{\ev^\fr_{j}}_{\Csi_1}}\circ u_{\eta 1})\times ({{\ev^\fr_{j}}_{\Csi_2}}\circ u_{\eta 2}): \cK_{1,2}\to  (\ocI\times\ocI)^M.$$ On the right  we again use shorthand where $\ocI$ stands for $\ocI(\cD/\cT^\tw)$ and all products are fibered over $\cT^\tw$. As in \cite[Section 5]{AGV}, we have a cartesian diagram

$$\xymatrix{\cK_\eta\ar^v[rr]\ar[d]_{\bev^\fr_\eta} && \cK_{1,2} \ar[d]^{\bev^\fr_{1,2}} \\
\ocI^M\ar[rr]^(0.4){\tilde\Delta}&& (\ocI\times \ocI)^M.
}$$
Here the map $\tilde\Delta$ sends $\ocI(\cD/\cT^\tw)$ to itself by the identity map on the left component, and by the map $\iota : \ocI(\cD/\cT^\tw) \to \ocI(\cD/\cT^\tw)$ inverting the band  on the right. Indeed an object of the fibered product consists of a pair of maps to the glued target along with an isomorphism of the restricted maps on the gerbe with band inverted. Since the glued curve is a pushout, such a pair of maps with isomorphism is precisely the data of a map from the glued curve, hence an object of $\cK_\eta$. This works for arrow as well.

We now have the following:
\begin{proposition}\label{P:vir-glue-tw}
$$[\cK_\eta]^\vir \ \ = \ \ \tilde\Delta^! [\cK_{1,2}]^\vir.$$
\end{proposition}

\begin{proof} Recall the perfect obstruction theory $\EE^{\bullet}_{\cK_{\eta}/\fT_{0,r}^{\tw,\spl}} \to \LL^{\bullet}_{\cK_{\eta}/\fT_{0,r}^{\tw,\spl}}$ defined in \ref{obsth}.
By \cite[Proposition 5.10]{Behrend-Fantechi} it suffices to produce a diagram of distinguished triangles 
$$  \xymatrix{
{v^*\EE^{\bullet}_{\cK_{1,2}/\fT_{0,r}^{\tw,\spl}}}^{\vphantom{\cK_{1,2}/\fT_{0,r}^{\tw,\spl}}} \ar[r]\ar[d]
& 
\EE^{\bullet\vphantom{\cK_{\eta}}}_{\cK_{\eta}/\fT_{0,r}^{\tw,\spl}} \ar[r]\ar[d]
& 
\bev_\eta^{\fr\, *}\, \LL_{\tilde\Delta} \ar[r]^(.65){[1]}\ar^{id}[d]
&\\
{v^*\LL^{\bullet}_{\cK_{1,2}/\fT_{0,r}^{\tw,\spl}}}^{\vphantom{\cK_{1,2}/\fT_{0,r}^{\tw,\spl}}} \ar[r] 
&\LL^{\bullet\vphantom{\cK_{\eta}}}_{\cK_{\eta}/\fT_{0,r}^{\tw,\spl}} \ar[r]
& \bev_\eta^{\fr\, *}\, \LL_{\tilde\Delta} \ar[r]^(.65){[1]}&.}
$$
Since $\tilde\Delta$ is a regular embedding $\LL_{\tilde\Delta}\simeq N_{\tilde\Delta}^\vee[1]$.
 
Consider the cartesian and co-cartesian square
$$\xymatrix{
\cG \ar[r]^{\iota_1}\ar[d]_{\iota_2} & \cC_1 \ar[d]^{\nu_1}\\ 
 \cC_1 \ar[r]^{\nu_2} & \cC
}$$ where $\cG$ is the disjoint union of the marking corresponding to the roots of $\Csi_1$ or $\Csi_2$. 
Also denote the normalization map $\nu: \cC_1\sqcup \cC_2 \to \cC$ and the embedding $\iota: \cG\to \cC$.
We have the standard normalization triangle
$$\xymatrix{
\LL_{\Box}^\vee \ar[r] & 
\nu_*\bL\nu^*\LL_{\Box}^\vee \ar[r] &
\iota_*\bL\iota^*\LL_{\Box}^\vee \ar[r]^(.65){[1]}&
}$$
and a natural decomposition 
$$\nu_*\bL\nu^*\LL_{\Box}^\vee \ \ =\ \ 
 \nu_{1*}\bL\nu_1^*\LL_{\Box}^\vee \ \oplus\  \nu_{2*}\bL\nu_2^*\LL_{\Box}^\vee.$$
\begin{lemma}
$$ \bL\nu_1^*\LL_{\Box}^\vee \ \ = \ \ \LL_{\Box_1}^\vee,\quad\quad \bL\nu_2^*\LL_{\Box}^\vee \ \ = \ \ \LL_{\Box_2}^\vee,$$
and
$$\bL\iota^*\LL_{\Box}^\vee \ \ = \ \ (f\circ \iota)^* T_\cD.$$
\end{lemma}
\begin{proof}[Proof of Lemma]
The commutative diagram 
$$\xymatrix{
\cC_1 \ar[r]\ar[d]& \cX_1\ar[d]\\
\cC\ar[r]& \cW
}$$
induces a canonical arrow $\bL\nu_1^*\LL_{\Box} \to \LL_{\Box_1}$, and similarly for $\bL\nu_2^*\LL_{\Box} \to \LL_{\Box_2}$. We can check that this is an isomorphism locally. Away from $\cD$ nothing is changed. Near $\cD$, the complex  $\LL_{\Box}$ is the conormal to $\cC \to \cW$  since $f$ is transversal, and it restricts to $\LL_{\Box_1}$, the conormal to $\cC_1 \to \cX_1$. For the same reason   the conormal to $\cC \to \cW$ restricts on $\cG$ to the conormal of $\cG\to \cD$.
\end{proof}

The triangle now looks as follows: 
$$\xymatrix{
\LL_{\Box}^\vee \ar[r] & 
\nu_{1*}\LL_{\Box_1}^\vee\oplus \nu_{2*}\LL_{\Box_2}^\vee  \ar[r] & \iota_*(f\circ \iota)^* T_\cD \ar[r]^(.65){[1]} &
}$$
Since $\cD$ is a gerbe, the tangent sheaf $T_\cD$ is the pullback of $T_D$, and it follows from the Tangent Bundle Lemma (see \cite[Lemma 3.6.1]{AGV}) that 
$$p_*\iota_*(f\circ \iota)^* T_\cD \ \ = \ \ \bev_\eta^{\fr\,*}\,N_{\tilde\Delta}.$$
Therefore when applying  $\bR\pi_*$, dualizing and rotating the above triangle we get
$$\xymatrix{
(\bR p_{1*}\LL_{\Box_1}^\vee)^\vee\oplus (\bR p_{2*}\LL_{\Box_2}^\vee)^\vee \ar[r] &(\bR p_{*}\LL_{\Box}^\vee)^\vee \ar[r] & 
 \bev_\eta^{\fr\,*}\,N_{\tilde\Delta}^\vee[1] \ar[r]^(.65){[1]} &
}$$
which is clearly compatible with the triangle of cotangent complexes,
as required. (A detailed verification of such compatibility is found in \cite[Appendix]{ACW}.) 
\end{proof}

\subsection{Comparison with $\Delta$}\label{Sec:coarse-diagonal}

We now translate Proposition \ref{P:vir-glue-tw} into a result involving $\Delta: \ocI(D)^{M} \to (\ocI(D)\times \ocI(D))^{M}$ instead of $\ocI^M = \ocI(\cD/\cT^\tw)_{\cT^\tw}^{M}$ and ${\tilde\Delta}$.

We have a cartesian diagram
$$
\xymatrix{
\cK_\eta \ar^{q_\eta}[r]\ar[d]& \cK_\eta^*\ar[rr]\ar[d] && \cK_{1,2}\ar[d]\\
\ocI(\cD)^{M}\ar^{q}[r]& {*} \ar[rr]\ar[d] && (\ocI(\cD)\times \ocI(\cD))^{M}\ar[d] \\
& \cT^\tw\times \ocI(D)^{M}\ar^(.4){Id\times \Delta}[rr]\ar[d] && \cT^\tw\times(\ocI(D)\times \ocI(D))^{M}\ar[d]\\
& \ocI(D)^{M}\ar^(.4){\Delta}[rr] && (\ocI(D)\times \ocI(D))^{M}.
}
$$
The arrow $\Delta$ is again the diagonal composed with $\iota$ on the right.

By Lemma \ref{L:eval-lift-c}, the component $\cZ^j$ of $\ocI(\cD/\cT^\tw)$ where $\ev^\fr_j$ maps is a gerbe over the corresponding component $Z^j$ of $\cT^\tw\times \ocI(D)$, and this gerbe is banded by $\bmu_{c_j}$. 
%
It follows that the arrows $q$ and $q_\eta$ are \'etale surjective of pure degree $\prod_{j\in M} c_j$: the arrow $q$ is the product of the \'etale surjective morphisms $\cZ^j \to  \cZ^j\times_{Z^j}\cZ^j = \cZ^j\times B\bmu_{c_j}$. Therefore we have the following:
\begin{proposition} \label{P:diag-gysin}
$$(q_{\eta})_*\,[\cK_\eta]^\vir \ \ = \ \ \left(\prod_{j\in M} c_j\right) \, \cdot \, \Delta^! [\cK_{1,2}]^\vir.$$
\end{proposition}

We can now use the projection formula. The composite top morphism in the last diagram is $v_\eta:\cK_\eta \to \cK_{1,2}$. We can compose the vertical arrow on the right and obtain the ``untwisted" evaluation morphism ${\bev}_{1,2}: \cK_{1,2} \to (\ocI(D)\times\ocI(D))^M.$ Denoting by $[\Delta]$ the class
$(\Delta)_*[\ocI(D)^M]$, we have that 
$$v_{\eta*}\,[\cK_\eta]^\vir \ \ = \ \ \left(\prod_{j\in M} c_j\right) \, \cdot \, {\bev}_{1,2}^* [\Delta]\cap [\cK_{1,2}]^\vir.$$
But
$[\Delta] = \prod_{j\in M}\left(\sum_{\delta_j \in F}
\delta_j \times \iota^*\delta_j^\vee\right).$
We thus obtain
\begin{corollary}\label{C:glue-virt}
	$$v_{\eta*}\,[\cK_\eta]^\vir \ \ = \ \ \prod_{j\in M} \left(c_j  \sum_{\delta_j \in F}
	 {{\ev_{j}}_{\Csi_1}}^{*} \delta_j \times {{\ev_{j}}_{\Csi_2}}^{*}   \iota^*\delta_j^\vee \right) \cap [\cK_{1,2}]^\vir;$$
	  combining with Lemma \ref{Lem:glue-target}, with a slight abuse of notation we have
	  \begin{eqnarray*}
\lefteqn{(v_{\eta}\circ u_\eta)_*\,[\cK_\eta]^\vir }\\
\ \ &=& \ \ r(\eta)\,\cdot\,\prod_{j\in M} \left(c_j  \sum_{\delta_j \in F}
	 {{\ev_{j}}_{\Csi_1}}^{*} \delta_j \times {{\ev_{j}}_{\Csi_2}}^{*}   \iota^*\delta_j^\vee \right) \cap [\cK_{\Csi_1}\times\cK_{\Csi_2}]^\vir
\end{eqnarray*}
\end{corollary}

\subsection{End of proof}\label{SS:endpf}
The stack $\cK_\eta$ carries two universal families of contracted curves: a disconnected family $ C'\to \cK_\eta$ pulled back from $ C_1 \sqcup C_2 \to \cK_{\Csi_1}\times \cK_{\Csi_2}$ inducing evaluations $\ev'_i$ with coarse curve $\underline C'\to \cK_\eta$ having sections $s_i'$;  and a connected family $C_\eta \to \cK_\eta$ coming  from $C \to \cK$ inducing evaluations $\ev_i$, with coarse curve  having sections $s_i$. These families differ only where they meet the splitting divisor. In particular  the pullback of the classes $\psi_i$ of the sheaves $s_i^*\omega_{\underline C/\cK}$ coincides with that of the class $\psi'_i$ corresponding to ${s'}_{i}^*\omega_{\underline C_2  \sqcup \underline C_2 /\cK_{\Csi_1}\times\cK_{\Csi_2}}$, and similarly for the pullbacks of $\gamma_i$ via evaluation maps.  We compute:

\begin{align*}
{\lefteqn{\left\langle{ \prod_{i\in N}\tau_{m_i}(\gamma_i)}\right\rangle_{\Gamma}^{W_0}}}\\
  & =  \sum_{\eta\in \Omega}\frac{r(\eta)}{|M|!}\deg\left((s\circ t_\eta)^*\left(\prod_{i\in N} \psi_i^{m_i}\cdot \ev_i^*\gamma_i\right)\cap {[\cK_\eta]^\vir}\right)\\
  \intertext{(by the projection formula and Lemma \ref{L:virt-split})}
   & =  \sum_\eta\frac{r(\eta)}{|M|!}\deg\left((u\circ v_\eta)^*\left(\prod_{i\in N} {\psi'}_i^{m_i}\cdot {\ev'}_i^*\gamma_i\right)\cap {[\cK_\eta]^\vir}\right)\\
   \intertext{(by the discussion above)} 
  & = \sum_{\eta\in \Omega}\frac{r(\eta)}{|M|!}\frac{\prod_{j\in M} c_j}{r(\eta)}\\
      & \quad\quad \deg\left(\left(\prod_{i\in N} {\psi'}_i^{m_i}\cdot {\ev'}_i^*\gamma_i\right)\right. \\ 
& \quad\quad\quad\quad\cdot  \left.   \prod_{j\in M} \left(c_j  \sum_{\delta_j \in F}
 {\ev_{j}}_{{\Csi_1}}^{*} \delta_j \times {\ev_j}_{\Csi_2}^{*}   \iota^*\delta_j^\vee \right)
      \cap{[\cK_{\Csi_1}]^\vir\times [\cK_{\Csi_2}]^\vir}\right)
\intertext{(by the projection formula and Corollary \ref{C:glue-virt})}
    & = \sum_{\eta\in\Omega}\frac{\prod_j c_j}{|M|!}
    \sum_{\delta_j\in F\forall j\in M} (-1)^{\epsilon}
 \left\langle \prod_{i\in N_1} \tau_{m_i}(\gamma_{i})\bigg\vert\prod_{j\in M}\delta_j
  \right\rangle^{(X_1,D)}_{\Csi_1}\\
  & \quad\quad\quad\quad \quad\quad\quad\quad \quad\quad\quad\quad \cdot
 \left\langle 
 \prod_{i\in N_2} \tau_{m_i}(\gamma_{i})\bigg\vert\prod_{j\in M}\tilde\delta^\vee_j
 \right\rangle^{(X_2,D)}_{\Csi_2}
  \end{align*}
as required.  

\appendix
\section{Pairs and nodes}\label{A:pairs}
\subsection{Smooth and locally smooth pairs}\label{onpairs}

A {\em smooth pair} is a pair $(X,D)$ where $X$ is a smooth algebraic stack and $D$ is a smooth divisor. A {\em locally smooth pair} is obtained if we only require $X$ to be smooth near the smooth divisor $D$. We sometimes call $X$ the {\em ambient} scheme/stack and $D$ the {\em boundary divisor}.

Let $\cA:=[\AA^1/\GG_m]$  - this notation will be kept throughout the paper. A morphism $f:X\to \cA$ is equivalent to the data $(L,s)$ of a line bundle $L$ on $X$ with a section $s$, as explained in \cite{ACW}. The morphism $f$ is dominant if and only if the section $s$ is nonzero; in particular, every pair $(X,D)$ with $X$ an algebraic stack and $D$ an effective Cartier divisor defines such a dominant morphism.  The pair is smooth (respectively locally smooth) if and only if the morphism to $\cA$ is smooth (respectively smooth over the divisor $B\GG_m=[0/\GG_m]$).

A morphism of locally smooth pairs $\phi:(X,D)\to (X',D')$ is a morphism $\phi:X\to X'$ such that $\phi^{-1}(D')_{red}\subset D$.
If $D'$ is empty, every morphism $X\to X'$ defines a morphism of pairs $(X,D)\to (X',\emptyset)$. 

A family of locally smooth pairs over a base stack $S$ is the datum of a flat morphism $X\to S$ and an $S$-flat closed substack $D\subset X$ such that for every point $s\in S$ the fiber $(X_s,D_s)$ is a locally smooth pair. 

Given two families $(X,D)$ and $(X',D')$ of locally smooth  pairs over the same base $S$, a log morphism is a morphism $\phi:X\to X'$ such that 
\begin{enumerate}
 \item for every $s\in S$, $\phi_s:(X_s,D_s)\to (X'_s,D'_s)$ is a morphism of locally smooth pairs;
\item the morphism $\phi^*\Omega_{X'}\to \Omega_X$ induces a morphism $\phi^*(\Omega_{X'}(\log D'))\to \Omega_X(\log D)$.
\end{enumerate}

Note that:
 \begin{enumerate}
 \item[(a)] The first condition implies the second if $S$ is reduced, but not in general;
\item [(b)] Assume that $S$ is connected, and that $D$ has connected components $D_i$ such that $D_i\cap D_s$ is also connected for every $i$ and every $s\in S$ (this is true e.g. if $S$ is simply connected). Then $\phi:X\to X'$ is a log morphism if and only if there exist nonegative integers $c_i$ such that $\phi^*D'=\sum c_iD_i$.
\end{enumerate}

Locally smooth pairs  their morphisms are classical special cases of logarithmic structures in the sense of \cite{Kato}.

\subsection{Transversality for nodal singularities and  pairs}
\label{logsmpair}

A morphism of locally smooth pairs $(C,\Sigma)\to (X,D)$ is {\em trans\-ver\-sal to the boundary divisor} (or just transversal) if the scheme theoretic inverse image of $D$ is smooth (and hence a union of connected components of $\Sigma$).

Let $X$ be a complex algebraic stack; we say that it has {\em nodal codimension one   singularities} - or just {\em nodal singularities} for brevity - if it is  locally isomorphic in the {f.p.p.f.} topology to  $\{xy=0\}\times\AA^n$; in particular its singular locus $D$ is smooth. Let $\nu:\tilde X\to X$ be the normalization and $\tilde D=\nu^{-1}(D)$. Then $(\tilde X,\tilde D)$ is a smooth pair,  and $\tilde D\to D$ is an \'etale double cover.

\begin{defn} A morphism between nodal algebraic stacks $f:C\to X$ is called {\em transversal to the singular locus} if \begin{enumerate}
\item the induced morphism $\tilde C\to \tilde X$ defines a morphism of locally smooth pairs which is transversal to the boundary divisor;
\item for every point $p\in f^{-1}(D)$ its two inverse images in $\tilde C$ map to different points of $\tilde D$ via $\tilde f$.
\end{enumerate}
\end{defn}

 This means that we have smooth charts $\tilde C \to C$ and $\tilde X \to X$,  lifting $\tilde C \to \tilde X$ of $C\to X$ and smooth morphisms $\tilde C \to \{xy=0\}$ and $\tilde X \to \{xy=0\}$ making the following diagram commutative
\[\xymatrix{
\tilde C \ar[rr]\ar[rd]&& \tilde X\ar[dl]\\
& \{xy=0\},
}
\]
so on the charts $\tilde C$ and $\tilde X$ the coordinates $x,y$ with $xy=0$ are the same.

Suppose now $(C,\Sigma)\to S$ and  $ (X,D) \to S$ are flat families of locally smooth pairs with at most nodal singularities and $f: C \to X$ a map. The following is evident:

\begin{lemma}\label{L:tr-open} The locus $S^{tr} \subset S$ where the fibers are transversal is open
\end{lemma}

\subsection{First-order smoothability of nodal singularities}\label{SS:1st-order-sm}
If $X$ is a  stack with codimension-1 singular locus $D$, we say that $X$ is {\em first-order smoothable} if the line bundle $\cExt^1(\Omega_X,\mathcal O_X)$ on $D$ is trivial.
If $X$ is the union of two smooth components $X_1$ and $X_2$ meeting transversally along $D$, then it is first-order smoothable if and only if $N_{D/X_1}\otimes N_{D/X_2}$ is isomorphic to $\cO_D$. 
Note that if there is a one-parameter smoothing of $X$ with smooth total space then $X$ is first-order smoothable, while the converse is in general not true.

\section{Stack constructions}\label{A:stacks}
\subsection{Using $2$-stacks to define stacks}\label{2stacks_as_stacks}

Occasionally we define a $2$--groupoid $\fX$ by giving objects, $1$-morphisms and $2$-morphisms, and then we show that every $1$-morphism in $\fX$ is rigid (i.e., it has only the identity as $2$-automorphism); equivalently, for any two objects $X$ and $Y$ of $\fX$, the groupoid $\Mor(X,Y)$ is equivalent to a set. In this case we say that the $2$-groupoid  $\fX$ is $1$-rigid. We can then consider {\em the associated groupoid $\fX^{[1]}$,} where objects are unchanged, and morphisms are isomorphism classes of  $1$-morphisms of the given 2-groupoid $\fX$. Since $\fX$ is 1-rigid, it is equivalent to $\fX^{[1]}$ (in the appropriate lax sense). We might as well replace $\fX$, which may arise naturally but is likely to intimidate us with its d\ae monic 2-arrows,   by the more friendly, yet  equivalent, groupoid $\fX^{[1]}$.

We will use this particularly in the definition of algebraic stacks. 
In particular if $X$ is a stack then stacks with a representable morphism to $X$ form a $1$-rigid $2$-groupoid, see Lemma 3.3.3 in \cite{AGV}; also, the $2$-groupoid of stacks with a dense open algebraic space and isomorphisms as $1$-morphisms is also $1$-rigid, 
see Lemma 4.2.3 on page 42 of \cite{Abramovich-Vistoli}.
Both cases are generalized using the following lemma.

\begin{lemma}\label{sepdiag}
Let $X$ be a stack with separated diagonal, and $U$ a scheme-theoretic dense open substack. Let $\beta:\id_X\to \id_X$ be a $2$--morphism such that $\beta|_U$ is the identity of $\id_U$. Then $\beta$ is the identity $2$--morphism.
\end{lemma}
\begin{proof}
 The fact that the diagonal is separated implies that the natural projection $\pi:I(X)\to X$ is separated. The automorphisms of $\id_X$ are the sections of $I(X)\to X$. Since we assumed that this section is the identity on a scheme-theoretically dense substack, it coincides with the identity on $X$. 
\end{proof}

Let $p:X\to Y$ be a morphism of stacks, and assume that there is a scheme-theoretic dense open substack $U$ of $X$ such that $p|_U:U\to p(U)$ is an isomorphism. Let $g:Y\to Y$ be an isomorphism. A lifting of $g$ to $X$ is a pair $(f,\alpha)$ such that $f:X\to X$ is an isomorphism and $\alpha:p\circ f\rightarrow  g\circ p$ a $2$-morphism. A morphism of liftings is a $2$-morphism $\gamma:f\to f'$ such that $\gamma$ and $\alpha'$ induce
$\alpha$.

\begin{lemma}\label{rigid_liftings} The groupoid of liftings of $g$ is rigid, i.e., equivalent to a set.
\end{lemma}
\begin{proof}
Let $(f,\alpha)$ and $(f',\alpha')$ be two liftings of $g$ to $X$. We want to show that a $2$-morphism $\gamma:f\to f'$ such that $\gamma$ and $\alpha'$ induce $\alpha$ is unique if it exists. 
 Let $\gamma$ and $\gamma'$ be two such two-morphisms, and let $\beta:=\gamma^{-1}\circ \gamma':f\to f$. Then $\beta|_U:f_U\to f_U$ is the identity $2$--morphism. Let $\bar\beta:\id_X\to \id_X$ be the composition of $\beta$ with the identity of $f^{-1}$; then $\bar\beta|U$ is also the identity $2$-morphism of $\id_U$. Therefore $\bar\beta$ must be the identity $2$--morphism by Lemma \ref{sepdiag}, and hence $\beta$ must be the identity $2$--morphism, hence $\gamma=\gamma'$.
\end{proof}

\begin{conv}\label{conv:rigid_liftings} In this case, we will refer to the liftings of $g$ as a set, meaning the set of equivalence classes of the corresponding groupoid.
\end{conv}
\subsection{Inertia stacks of various flavors}
\subsubsection{The inertia stack}\label{SS:inertia} Let $X$ be an algebraic stack. Its {\em inertia stack} $\cI(X)$ is the stack whose objects over a scheme $S$ are pairs $(x,g)$  with $x\in X(S)$ and $g\in \Aut(x)$. Arrows are  given by pullback diagrams.

The inertia stack can be identified as $\cI(X) = X \times_{X\times X} X$, with both arrows given by the diagonal. Since the diagonal is representable, the morphism $\cI(X) \to X$ given by the first projection is representable. This is simply the forgetful morphisms which sends an object $(x,g)$ to $x$.

Let $B\ZZ$ be the classifying prestack of $\ZZ$. Then we have a canonical isomorphism of  prestacks $\cI(X) \simeq Hom(B\ZZ,X)$. This in particular implies that forming the inertia stack is compatible with fiber products: given a fiber product of algebraic stacks  $\cX = \cX_1 \times _\cZ \cX_2$ we have 
$\cI(\cX) = \cI(\cX_1) \times _{\cI(\cZ)} \cI(\cX_2)$ (an observation due to Tom Bridgeland).

\subsubsection{Inertia of Deligne--Mumford stacks} Suppose now $X$ is a  Deligne--Mumford stack, and let $r$ be a positive integer such that the exponent of any automorphism group in $X$ divides $r$. In this case we have $\cI(X) = Hom(B(\ZZ/r\ZZ),X)$. The stack $\cI(X)$ has an evident decomposition $\cI(X) = \sqcup_{d|r}\cI_d(X)$, where $\cI_d(X)$ is the stack of $(x,g)$ with $g$ of order $d$. Then  $\cI_d(X) = Hom^{rep}(B(\ZZ/d\ZZ),X)$, the substack of {\em representable} morphisms, see \ChDan{\cite[Definition 3.2.1]{AGV}}.

\subsubsection{Rigidified inertia} The automorphism group of an object $(x,g)$ of $\cI_d(X)$ has the subgroup $\ZZ/d\ZZ \simeq \langle g \rangle$ sitting in its center. We can therefore rigidify by removing this subgroup and obtain the rigidified stack $\ocI_d(X) = \cI_d(X) \thickslash (\ZZ/d\ZZ)$. It is canonically isomorphic to the stack whose objects over $S$ are $\cG \to X$, where $\cG$ is a gerbe banded by $\ZZ/d\ZZ$ and $\cG \to X$ is representable. The {\em rigidified inertia stack} is $\ocI(X) =  \sqcup_{d|r}\ocI_d(X)$. The morphism $\cI(\cX) \to \ocI(X)$ is the universal gerbe, with universal representable morphism $\cI(X) \to X$. We stress that the data of the band is important - without it we would get a different tack, a rigidification of the stack of cyclic subgroups (without choice of generator) of inertia.

\subsubsection{Cyclotomic inertia and rigidified inertia}
In the theory of twisted stable maps, a cyclotomic twist of these stacks arises naturally. Since in this paper we work over $\CC$, it is safe to choose the generator $\exp(2\pi\,i/d)$ of $\bmu_d$, so the distinction is not crucial. Let us mention the appropriate identification of stacks: we have $\cI(X) \simeq Hom(B\bmu_r,X)$, the cyclotomic inertia stack; $\cI_d(X)\simeq Hom^{rep}(B\bmu_d,X)$; and  $\ocI_d(X) \simeq \cI_d(X) \thickslash (\bmu_d)$ is canonically isomorphic to the stack whose objects over $S$ are $\cG \to X$, where $\cG$ is a gerbe banded by $\bmu_d$ and $\cG \to X$ is representable. The stack  $\ocI(X) =  \sqcup_{d|r}\ocI_d(X)$ is then identified as the {\em \ChDan{rigidified} cyclotomic inertia stack}, see \ChDan{\cite[Section 3.4]{AGV}.}

\subsection{Deformations and obstructions for Artin stacks}

A key technical tool for deformation theory is the cotangent complex of a morphism: we refer the reader to \cite{Olsson-cc} for the correct definition of cotangent complex $\LL_f$ for a morphism $f:\cX \to \cY$ of Artin stacks and for the relevant results in deformation theory, see also \cite{Aoki-def}. Note that in \cite{Olsson-cc} Olsson's cotangent complex $\LL_f$ is actually not defined as an object in the derived category:  its right truncations  $\tau_{\ge n}\LL_f$ are  for $n\in \ZZ$, and $\LL_f$ is defined as an object of a filtered category. This issue is removed in \cite[2.2.ix]{Laszlo-Olsson}, specifically the equivalence at the end of page 119 between the appropriate derived categories of quasi-coherent sheaves on the stack and on a symplicial resolution. 

In particular,      to any morphism of Artin stacks $f:X\to Y$ we can after all associate its cotangent complex $\LL_f\in D^{\le 1}_{coh}(X)$. 
%
This is functorial, in the sense that for any composable morphisms of Artin stacks $f:X\to Y$ and $g:Y\to Z$, there is a distinguished triangle in $D^-(X)$:
 $$
f^*\LL_g\to \LL_{g\circ f}\to \LL_f\stackrel{+1}{\to}.$$

The morphism $f$ is \'etale if and only if $\LL_f=0$; it is smooth if and only if its cotangent complex $\LL_f$ is perfect of perfect amplitude contained in $[0,1]$.

Recall that $f$  is said to be {\em of Deligne--Mumford type} if for any morphism $S\to Y$ with $S$ a Deligne--Mumford stack, the stack $X\times_YS$ is also Deligne--Mumford. Then $f$ is Deligne--Mumford type if and only if $h^{1}(\LL_f)$ is the zero sheaf, or equivalently if $\LL_f\in D^{\le 0}(X)$.

For any stack $X$ we write $\LL_X$ for $\LL_{X\to \Spec k}$; the complex $\LL_X$ is perfect of perfect amplitude in $[-1,0]$ if and only if $X$ is a Deligne--Mumford stack with l.c.i. singularities.

If $f:X\to Y$ is a morphism of Deligne--Mumford type, an {\em obstruction theory} for $f$ is a morphism $\phi:\EE\to \LL_f$ in $D^{\le 0}(X)$ such that $h^0(\phi)$ is an isomorphism, and $h^{-1}(\phi)$ is surjective. We say that it is a {\em perfect obstruction theory} if $\EE$ is a perfect complex, of perfect amplitude contained in $[-1,0]$ (i.e., locally isomorphic to a morphism $E^{-1}\to E^0$ of locally free sheaves).

We define the cotangent complex of a locally smooth pair $(X,D)$ to be $\LL_{X/\cA}$ (where $X\to \cA$ is the morphism associated to the pair, see \S \ref{onpairs}); we sometimes denote it by $\LL_{X(\log D)}$. Note that if $(X,D)$ is a smooth pair with $X$ a scheme, or, more generally, a smooth DM stack, then $\LL_{X(\log D)}$ is concentrated in degree zero and isomorphic to the classically defined locally free sheaf $\Omega_X(\log D)$. It is easy to see that a morphism of locally smooth pairs induces a morphism of log cotangent complexes, which has the usual deformation-theoretic properties (see \cite{Olsson-log-cotangent}).

 \section{Stacks of maps and their obstruction theory}
\label{obs_th}
\subsection{Stacks of maps}
We  define a relative obstruction theory on certain algebraic stacks parametrizing stable maps. This includes the obstruction theories needed in the singular and in the relative case, see Section \ref{SS:obs-th}; in fact, a common generalization is possible.

\begin{conv}\label{flsp} In this section, we will fix an algebraic stack $T$, and a family of locally smooth pairs $(W,D)\to T$ such that $W\to T$ is of Deligne-Mumford type (note that the case $D=\emptyset$ is possible, in which case we are just assuming $W\to T$ to be flat). Fix nonnegative integers $g,n$ and a curve class $\beta$ in the fibers of $W\to T$. Fix $n$-tuples ${\bf e}$ of positive integers $e_i$ and  ${\bf c}$ of nonnegative integers $c_i$ such that $\sum c_i\cdot {e_i}^{-1}=\beta\cdot D$. In particular if $D=\emptyset$, we must have $c_i=0$. We combine the data under the shorthand notation $\Gamma=(g,n,\bf e,\bf c,\beta)$
\end{conv}

\begin{definition} \label{strepmap} Let $\tfK_{\Gamma}((W,D)/T)$ be the stack\Barbara{Do we need more details?} of representable maps $f$ from a twisted prestable $n$-pointed curve  $(C,\Sigma)$ to fibers of $(W,D)\to T$ such that $f^*D=\sum c_i\Sigma_i$ and such that $\Sigma_i$ is twisted with index $e_i$. If $D$ is empty, we write $\tfK_{\Gamma}(W/T)$. 
\end{definition}

\begin{remark}
 The condition on $f^*D$ can be rephrased as saying that we consider the stack of log morphisms; see section \ref{onpairs} for details.
\end{remark}
\begin{conv} 
We will write just $\tfK$ for $\tfK_{\Gamma}((W,D)/T)$ within this section. 
\end{conv}

\begin{lemma}\label{strepmapalg} The stack $\tfK$ is an algebraic stack in the sense of Artin.
\end{lemma}

\begin{proof} We first do the case $D=\emptyset$. Consider the stack of twisted curves $\fM :=\fM^\tw_{g,n}$ and its universal family $\fC \to \fM$. Over $\fM \times T$ we have two families, $\fC' := \fC \times T \to \fM \times T$ and $W' := \fM \times W \to \fM \times T$. We first prove that  $\Hom^{rep}_{\fM \times T}(\fC', W')$ is algebraic. By \cite{Olsson-hom}, the pullback of $\tfK$ to a scheme by any arrow $S\to \fM \times T$ is an algebraic stack in the sense of Artin. Also $\tfK$ is a stack, since this property is tested over a base scheme. By \ChDan{\cite[Lemma C.5]{AOV-maps}}, the stack $\tfK$ is an algebraic stack, as required. 

If $D$ is nonempty, let $\tfK'$ be the stack obtained by assuming $D$ empty. The stack $\tfK$ is obtained by first passing to the open substack where $f^*D$ is a divisor on $C$, and then to the closed substack where the two divisors $\sum c_i\Sigma_i$ and $f^*D$ coincide.\end{proof}

\begin{notation}\label{N:fK} Let us now denote by $\fK \subset \tfK$ the maximal open substack of $\tfK$ such that the morphism $\fK\to T$ is of Deligne--Mumford type. We will reserve the notation $K$ for the substack which is Deligne--Mumford in the absolute sense.
\end{notation}

\subsubsection{Base change}\label{fKbasechange} The construction of $\tfK$ and $\fK$ behaves well under base change in the following sense. Assume that $(W,D)\to T$ satisfies the assumption in Convention \ref{flsp}. Let $a_T:T'\to T$ be any morphism, and write $W':=W\times_TT'$ and $D':=D\times_TT'$. Let $\beta'$ be the homology class in the fibers of $W'\to T'$ induced by $\beta$, let $\Gamma'$ e obtained by replacing $\beta$ by $\beta'$ in $\Gamma$, and $a_W:W'\to W$  the natural morphism.  Then $(W',D')\to T'$ satisfies the same assumptions, and there is a natural cartesian diagram $$
 \xymatrix{
\tfK_{\Gamma'}((W',D')/T')\ar[r]^{a_{\tfK}}\ar[d] &\tfK_{\Gamma}((W,D)/T)\ar[d]\\
T'\ar[r]_{a_T} &T
}
$$
where $a_{\tfK}$ is given by mapping an object $(C,\Sigma,f')$ to $((C,\Sigma, a_W\circ f)$. 
Since the property of being of DM type is stable under base change, one gets an analogous cartesian diagram by replacing $\tfK$ by $K$.

\subsubsection{Change of family}\label{fKpropermap} Assume that we are given a proper morphism $\theta: W_1\to W$ and a closed substack $D_1\subset W_1$ such that \begin{enumerate}
 \item the composite morphism $(W_1,D_1)\to T$ satisfies the assumption in Convention \ref{flsp};
\item one has $\theta^{-1}(D)=D_1$ as a closed subscheme, and $\theta|_{D_1}:{D_1}\to D$ is an isomorphism. 
\end{enumerate}
Let $\beta_1$ be a class in the fibers of $W_1\to T$, and $\beta=\theta_*\beta_1$, and $\Gamma_1,\Gamma$ the corresponding discrete data. 
If $\beta_1= 0$, assume moreover that $2g-2+n>0$. Then there is a natural induced proper morphism of $T$--stacks $$
\tfK_{\Gamma_1}((W_1,D_1)/T)\to \tfK_{\Gamma}((W,D)/T)
$$
defined by $(C,\Sigma, f)\mapsto (C,\Sigma, \theta\circ f)^{stab}$. This follows by applying \cite{Abramovich-Vistoli}, Corollary 9.1.3, where we replace the base scheme ${\mathbb S}$ by $T$ using \ChDan{\cite[Lemma C.5]{AOV-maps}.}

\subsection{Obstruction theory on stacks of maps}\label{SS:obs-th} For simplicity we now restrict to the open substack $\fK^\pitchfork\subset \fK$ parametrizing maps which are transversal to the boundary divisor in the sense of \ref{logsmpair}.   One can avoid this simplifying assumption using logarithmic structures, but we will not need this generality in this paper.

The aim of this section is to define a relative obstruction theory for $\fK^\pitchfork\to T$, and to give conditions so that it is perfect in $[-1,0]$. The construction works for $\tfK^\pitchfork$ instead of $\fK^\pitchfork$ if we allow obstruction theories for morphisms which are not of Deligne--Mumford type using the work \cite{Noseda}, requiring $\EE_{\tfK^\pitchfork/T} \to \LL_{\tfK^\pitchfork/T}$ to also be an isomorphism in degree $+1$.

Consider the structure commutative diagram
$$
\xymatrix{C \ar^f[drr]\ar_p[ddr]\ar[dr]&&\\
&W_{\fK}\ar[d]\ar_u[r]& W\ar[d]\\
&\fK\ar[r] & T
}
$$
where $C\to \fK^\pitchfork$ is the universal curve, $f$ is the universal map and $\Sigma:=\cup\Sigma_i$ is the union of the marked gerbes. 
Also denote $\Sigma' = \Sigma\setminus f^{-1} D$. 
Since we are assuming the maps are transversal to $D$, deforming $f: (C,\Sigma) \to (W,D)$ is equivalent to deforming $f: (C,\Sigma') \to W$, which is in turn equivalent to deforming the diagonal map $\bar f: (C,\Sigma') \to W_{\fK^\tw}$.

Let $\LL_{\Box}$ be the cotangent complex to the morphism $\bar f:(C,\Sigma')\to W_{\fK^\pitchfork})$: it is canonically  isomorphic in the derived category to 
the cone of the canonical morphism of cotangent complexes
$$f^* \LL_{W/T} \longrightarrow \LL_{C(\log \Sigma')/{\fK}}$$
induced by the structure morphisms $u^*\LL_{W/T}\to \LL_{W_{\fK^\pitchfork}/{\fK^\pitchfork}}$  - which is an isomorphism since $W\to T$ is flat -  and $\bar f^*\LL_{W_{\fK^\pitchfork}/{\fK^\pitchfork}}\to \LL_{C(\log\Sigma')/\fK^\pitchfork}$. See \cite[Theorem 8.1]{Olsson-cc}.

By the same argument, the object $\LL_{\Box}$ is also isomorphic to the cone of the morphism 
$$p^* \LL_{{\fK^\pitchfork}/T} \longrightarrow \LL_{C(\log\Sigma')/W}$$
and therefore there is a natural morphism 
$$\LL_{\Box}[-1] \longrightarrow p^* \LL_{{\fK}/T}.$$

The morphism $p$ is proper. It is also l.c.i., therefore Gorenstein, and its dualizing complex $\omega_p$ is a line bundle positioned in degree $-1$.
 By \cite{LN07}
 the functor $\bR p_*:\cD(C)\to \cD({\fK})$ has a right adjoint $p^!$ which is isomorphic to the functor $\FF\mapsto \bL p^*(\FF)\otimes \omega_p$.

We denote by $\EE_{{\fK}/T}$ the object  $\bR p_*({\LL_{\Box}}\otimes \omega_p)$; there is a natural morphism $\EE_{{\fK}/T}\to \LL_{{\fK}/T}$ induced by adjunction from the morphism $\LL_{\Box}[-1]\to p^*\LL_{{\fK}/T}$ defined above. Note that we have a canonical isomorphism $\EE_{{\fK}/T}\simeq (\bR p_*\LL_{\Box}^\vee)^\vee[-1].$

\begin{lemma}\label{obsth}
(1) The morphism $\EE_{{\fK}/T} \to \LL_{{\fK}/T}$ is an obstruction theory. 

(2) Its formation commutes with base change on $T$ in the sense of Remark \ref{fKbasechange}.
\end{lemma}
\begin{proof} We first give a proof in the case where $\Sigma=\emptyset$. 
(1) We use the criterion in \cite{Behrend-Fantechi}, Lemma 4.5, see also p. 85 there for the relative case - a detailed proof is available in \cite{ACW}.
The result then follows from \cite{Illusie}, Theorem
III 2.1.7.

(2) Given a morphism $\phi: \fK' \to {\fK}$ with $C' \to \fK'$ the pullback of $C \to {\fK}$, we have a canonical isomorphism 
$\bL\phi^* \EE_{{\fK}/T} \simeq \EE_{\fK'/T}$, such that  the composite morphism 
$\bL\phi^* \EE_{{\fK}/T} \to \bL\phi^* \LL_{{\fK}/T} \to \LL_{\fK'/T}$ 
coincides with the composition 
$\bL\phi^* \EE_{{\fK}/T} \to \EE_{\fK'/T} \to \LL_{\fK'/T}$. 
In particular, given a morphism $\psi: T' \to T$ we can pull back the entire diagram. Denote by $\phi: {\fK}' \to {\fK}$ the pullback via $\phi$. Again we have an isomorphism   $\bL\phi^* \EE_{{\fK}/T} \simeq \EE_{{\fK}'/T}\simeq \EE_{{\fK}'/T'}$, and the compatibility above  lifts to  $\LL_{{\fK}'/T'}$.

For the general case, we remark that the above proof remains valid by replacing the cotangent complex of $C$  by the cotangent complex with logarithmic poles along $\Sigma$. Treatment of this can be found in \cite{Laudal,Ran-pairs}. This also follows from \cite[III, \SS2.3 and \S4]{Illusie} by using the cotangent complex of the topos $\Sigma'C$, or by using the morphism $C \to \cA$ associated to $\Sigma'$.\end{proof}

\subsection{Perfect amplitude}
\begin{definition}
Let $t$ be a geometric point of $T$, $W_t$ the fiber of $W$ over $t$.  We say that a prestable map $f:(C,\Sigma_i)\to W_t$ (i.e., a point in $\fK$) is {\em nondegenerate} if 
no irreducible component of $C$ maps to the singular locus of $W_t$. An irreducible component which does map to the singular locus of $W_t$ is callled {\em degenerate}.
\end{definition}
\begin{remark} The points corresponding to nondegenerate maps form an open substack $\fK_{nd}$ of $\fK$, which commutes with base change in the sense of Remark \ref{fKbasechange}.
\end{remark}
\begin{lemma}\label{itsperf} Assume that the morphism $W\to T$ is l.c.i.

(1) The obstruction theory $\EE_{\fK^\pitchfork/T}$ is perfect in $[-2,0]$;
 
(2) It is perfect in $[-1,0]$  over the open substack $\fK^\pitchfork_{nd}$. 
\end{lemma}
\begin{proof} (1) Since $W\to T$ and $C\to \fK$ are l.c.i., both $\LL_{W/T}$ and $\LL_{C(\log \Sigma')/{\fK}}$ are perfect in $[-1,0]$. Therefore $\LL_{\Box}$ is perfect in $[-2,0]$; hence $\LL_{\Box}\otimes\omega_\pi[-1]$ is also perfect in $[-2,0]$, and since $p$ is proper and flat of  relative dimension $1$, one has that $\EE_{{\fK^\pitchfork}/T}={\bf R}p_*(\LL_{\Box}\otimes\omega_p[-1])$ is perfect in $[-2,1]$. Since $\EE_{{\fK^\pitchfork}/T}$ is an obstruction theory, it has vanishing $h^1$, hence it is perfect in $[-2,0]$.

(2) It is enough to prove that, for each point $x\in {\fK^\pitchfork_{nd}}$, $h^{-2}(x^*\EE_{{\fK^\pitchfork_{nd}}/T})=0$. Assume the point $x$ corresponds to a prestable map $f:(C,\Sigma)\to W$. We want to show that $\bExt^2(\LL_{\Box}|_C, \cO_C)=0$; by the local-to-global spectral sequence of $\Ext$, we reduce to showing that $H^1(C,h^{1}((\LL_{\Box}|_C)^\vee))=0$. Remark that $\LL_{\Box}|_C$ is the mapping cone of the morphism  $f^*\LL_{W/T} \to \LL_{C(\log \Sigma)}$.
Note that the support of $h^1((\LL_{\Box}|_C)^\vee)$  is contained in the locus of points in $C$ which map to the singular locus of $W\to T$ (i.e., the support of $h^1((f^*\LL_{W/T})^\vee)$) which by assumption is zero-dimensional.
\end{proof}
\begin{remark} In fact, both in this subsection and in the following one we could replace the moduli stack of twisted prestable curves with any other moduli stack of $d$-dimensional proper Deligne--Mumford  locally smooth pairs; Lemma \ref{obsth} would still hold, and Lemma \ref{itsperf} would hold with $[-2,0]$ (respectively $[-1,0]$) replaced by $[-(d+1),0]$ (respectively $[-d,0]$).
\end{remark}

\end{document}